\documentclass[12pt]{article}
\usepackage{amssymb}
\usepackage{amsmath}
\usepackage{amsthm}
\usepackage[colorlinks=true, allcolors=blue]{hyperref}
\usepackage{dsfont}
\usepackage{multirow}
\usepackage{booktabs}
\usepackage{subfigure}
\usepackage{authblk}
\usepackage{graphicx}
\usepackage{enumitem}
\usepackage[letterpaper,top=3cm,bottom=2cm,left=2.5cm,right=2.5cm,marginparwidth=1.75cm]{geometry}
\usepackage[authoryear, round]{natbib}
\def\dd{{\rm d}}

\def\ee{{\rm e}}
\def\TT{{\rm T}}

\def\mb{\mathbb}
\def\tod{\stackrel{d}{\to}}
\def\nablaq{{^{q}\nabla}}

\numberwithin{equation}{section}
\newtheorem{theorem}{Theorem}
\newtheorem{lemma}[theorem]{Lemma}

\newtheorem{corollary}[theorem]{Corollary}
\newtheorem{condition}{Condition}
\newtheorem{assumption}{Assumption}
\theoremstyle{definition}

\theoremstyle{remark}
\newtheorem{remark}{Remark}
\newcommand\keywords[1]{\textbf{Keywords}: #1}
\newcommand\msccode[1]{\textbf{2020 MSC}: #1}

\begin{document}

\title{The continuous-time limit of quasi score-driven volatility models \thanks{This work is partially supported by the National Natural Science Foundation of China (Grant No. 12371144) and Fundamental Research Funds for the Central Universities (Grant No. CXJJ-2024-439).}}

\date{}

\author{Yinhao Wu$^{a, }$\thanks{Corresponding author. Email: \href{mailto:yinhaowu@stu.sufe.edu.cn}{yinhaowu@stu.sufe.edu.cn}}}
\author{Ping He$^{a}$}

\affil{$^a$\it School of Mathematics, Shanghai University of Finance and Economics, 200433, Shanghai, China}

\maketitle

\begin{abstract}

This paper explores the continuous-time limit of a class of Quasi Score-Driven (QSD) models that characterize volatility. As the sampling frequency increases and the time interval tends to zero, the model weakly converges to a continuous-time stochastic volatility model where the two Brownian motions are correlated, thereby capturing the leverage effect in the market. Subsequently, we identify that a necessary condition for nondegenerate correlation is that the distribution of driving innovations differs from that of computing score, and at least one being asymmetric. We then illustrate this with two typical examples. As an application, the QSD model is used as an approximation for correlated stochastic volatility diffusions and quasi maximum likelihood estimation is performed. Simulation results confirm the method's effectiveness, particularly in estimating the correlation coefficient. 

\end{abstract}


\msccode{60F17, 60J25, 91B70}

\section{Introduction}
Since \cite{engle1982autoregressive} discovered conditional heteroskedasticity in economic data, time series models with time-varying parameters, especially variance, have gained increasing attention in financial research, with the GARCH model being the most well-known \citep{bollerslev1986generalized}. Numerous extensions within the GARCH family have since been developed \citep[e.g.,][]{nelson1991conditional,glosten1993relation,ding1993long}. Essentially, they are all observation-driven models, where the updating equations for parameters are based on observations. For example, the GARCH(1,1)-M model for the log asset price $X$ is given by:
\begin{equation}\label{garch}
    \begin{split}
        &X_n=X_{n-1}+cv_n+y_n,\\
        &v_{n+1}=\omega+\beta v_n+\alpha y_n^2,
    \end{split}
\end{equation}
where $y_n=\sqrt{v_n}\varepsilon_n$ is the innovation of $X_n$, $\{\varepsilon_n\}\stackrel{i.i.d.}{\sim} N(0,1)$, and $v_n$ is actually the conditional variance of the innovation. The variance update depends on the squared magnitude of current returns---large returns lead to higher future variance. However, sometimes such large deviations may just be accidental rather than systematic, implying that the future variance should not increase excessively. This issue is more pronounced when $\varepsilon$ follows a heavy-tailed distribution, such as in the t-GARCH model with Student's t-distributed innovations \citep{bollerslev1987conditionally}. Empirical evidence suggests that the prediction of its variance often results in overestimation \citep{laurent2016testing}.

Based on this consideration, \cite{harvey2008beta} and \cite{creal2013generalized} proposed the Generalized Autoregressive Score (GAS) model independently. The key improvement of this model lies in the parameter updating equation is driven by $\nabla_n$, the derivative of the log conditional density of innovation with respect to the parameter. This derivative, known as the {\it score}, gives the model its alternative name, the Score-Driven (SD) model. 
The model not only generalizes and enhances many time series
models but also demonstrates good empirical performance in fields such as economics, finance, and biology. Recently, it has been shown to be optimal in reducing the global Kullback-Leibler divergence between the true distribution and the postulated distribution \citep{gorgi2023optimality}. Within just a few years, more than 300 papers related to the SD model have been published (see \url{http://www.gasmodel.com/gaspapers.htm}).
 
Recently, \cite{blasques2023quasi} proposed the Quasi Score-Driven (QSD) model, where the score is no longer restricted to the conditional density of the observed innovations. This means the distributions driving the innovation and computing the score can differ. For instance, one can utilize a normal distribution to model innovations while employing the score of a Student's t-distribution to update parameters \citep[as in][]{banulescu2018volatility}, or vice versa, as in t-GARCH models. Therefore, QSD models encompass not only SD models but also other existing models, providing more flexibility.

On the other hand, continuous-time models are favored in theoretical finance research, particularly for derivatives pricing. These models, including the well-known Black-Scholes model,  describe the dynamics of underlying assets through one or a set of stochastic differential equations (SDEs). Among them, the counterparts for addressing time-varying variance include the Heston model \citep{heston1993closed}, the 3/2 model \citep{heston1997simple}, the 4/2 model \citep{grasselli20174} and so on. These are all stochastic volatility models, which assume that the volatility of asset follows a (typically mean-reverting) SDE.

\cite{nelson1990arch} was the first to bridge the gap between discrete and continuous-time models in finance. Specifically, he demonstrated that under certain scaling conditions, the continuous-time limit of equation \eqref{garch} is exactly the following stochastic volatility model:
\begin{equation}\label{cotinuegarch}
	\left\{\begin{aligned}
		&\dd X_t=cv_t\dd t+\sqrt{v_t}\dd W_t^{(1)},\\
		&\dd v_t=(\omega-\theta v_t)\dd t+\alpha v_t\dd W_t^{(2)},\\
		&\mbox{Cov}(\dd W_t^{(1)}, \dd W_t^{(2)})=0.
	\end{aligned}\right.
\end{equation}
In this sense, \eqref{cotinuegarch} is also known as the continuous-time GARCH model in some financial literatures. More generally speaking, as the sampling frequency increases and the time interval of a discrete-time Markov process goes to zero, it weakly converges to a diffusion process. The related research goes back to \cite{stroock1997multidimensional}, \cite{kushner1984approximation}, \cite{ethier1986markov}. Therefore, it can be seen that discrete-time models are intricately related to its approximate continuous-time counterpart. A useful insight is that by estimating the parameters of a discrete-time model, one can recover the parameters of the continuous-time process it approximates. This idea is known as Quasi Approximate Maximum Likelihood (QAML) estimation \citep[see e.g.][]{barone2005option, fornari2006approximating, stentoft2011american, hafner2017weak}. 

In the context of the SD model, \cite{buccheri2021continuous} explored the continuous-time limit of SD volatility models, obtaining a bivariate diffusion where the two Brownian motions are independent. While this result recover Nelson's limit in the case of normal density, it actually fails to capture the well-known Heston model. In the Heston model, the Brownian motions that drive returns and volatility are (negatively) correlated, a key feature that characterizes the leverage effect in the market.

In this paper, we investigate the continuous-time limit of QSD volatility models and also obtain a bivariate diffusion. However, unlike the result of the SD volatility models, the two Brownian motions in the bivariate diffusion can be correlated. It is shown that for a nondegenerate correlation coefficient, a necessary condition is that the distributions driving the innovation differ from that of computing the score, and at least one being asymmetric. Consequently, using QAML estimation based on the QSD model, we can recover the parameters of stochastic volatility models with leverage effects, particularly the correlation coefficient.

The paper is organized as follows. Section \ref{QSD volatility models} introduces the QSD model we are interested in. Section \ref{Main results} presents the main convergence results along with the analysis of the findings. Meanwhile, the proofs of main theorems are provided in Appendix. In Section \ref{Two examples}, we illustrate two specific examples of QSD volatility models, QSD-T and QSD-ST \citep[see][]{blasques2023quasi}, and explore some properties of their continuous-time limit through numerical simulation. Section \ref{QAML} presents Monte Carlo experiments for QAML estimation and filtering using the QSD model when the data-generating process (DGP) is a correlated volatility diffusion. Finally, Section \ref{Conclusion} concludes the paper.

\section{QSD volatility models}\label{QSD volatility models}

Let $\{y_n\}_{n\in\mb{N}}$ denote a time series of asset log returns, where $\mathcal{F}_n=\sigma(y_n,y_{n-1},\ldots, y_0)$ is the $\sigma$-algebra generated by $y$ up to time $n$. We assume that $\{y_n\}$ has the following form:
\begin{equation}\label{yeq}
    y_n=\sqrt{\varphi(\lambda_n)}\varepsilon_n, \quad \varepsilon_n|\mathcal{F}_{n-1}\stackrel{d}{\sim}f(\cdot,\Theta). 
\end{equation}
Here, $f(\cdot,\Theta)$ is a probability density function with a zero mean, and $\Theta$ represents the distribution's parameters \footnote{Hereafter, we omit $\Theta$ and treat probability density functions from the same family but with different parameters $\Theta_1, \Theta_2$ as distinct functions, denoted by $f_{\Theta_1}(\cdot)$, $f_{\Theta_2}(\cdot)$, or simply $f(\cdot), g(\cdot)$.}. Let $\lambda_n \in \mb{R}$ be a time-varying parameter, and in what follows, we will focus on designing the update rule that governs its evolution. The function $\varphi: \mb{R}\to\mb{R}^+$ is monotonic and differentiable, referred to as the link function. In this case, the conditional density of $y_n$ belongs to the class of scale family densities: 
\[p(y_n|\mathcal{F}_{n-1}; \lambda_n)=\frac{1}{\sqrt{\varphi(\lambda_n)}}f\left(\frac{y_n}{\sqrt{\varphi(\lambda_n)}}\right).\]
Equation \eqref{yeq}, together with the updating equation for $\lambda_n$, is commonly referred to as volatility models because, if $\varepsilon$ has unit variance, then $\sqrt{\varphi(\lambda_n)}$ behaves like the conditional volatility of $y_n$. Moreover, if $\varphi$ is the identity mapping restricted to $\mb{R}^+$, the time-varying parameter $\lambda_n$ is exactly conditional variance. Therefore, focusing on the QSD volatility model means restricting the conditional density of $y_n$ to the scale family densities.

The core of the QSD model lies in the time-varying parameter $\lambda_n$, which satisfies the following updating equation:
\begin{equation}\label{peq}
    \lambda_{n+1}=\omega+\beta\lambda_{n}+\alpha \psi(y_n,\lambda_n),
\end{equation}
where $\omega,\beta,\alpha\in\mb{R}$ are the parameters of equation. The key novelty in this expression is the term $\psi(y_n,\lambda_n)$, which equals $y_n^2$ in the GARCH model and $S(\lambda_{n})\nabla_{n}$ in the SD model, among others. In the SD model, the score $\nabla_{n}$ represents the partial derivative of the log conditional density of the observed innovations with respect to the parameter, i.e. 
\[\nabla_{n}=\frac{\partial \log p(y_n|\mathcal{F}_{n-1}; \lambda_n)}{\partial \lambda_n}.\] 
This term is analogous to the gradient in gradient descent algorithms. $S(\cdot)$ is a continuous function, referred to as the scaling function for the score. In order to interpret the curvature of the log density function, \cite{creal2013generalized} sets $S(\lambda_n)=[\mathbb{E}(\nabla_n^2|\mathcal{F}_{n-1})]^{-a}$, the inverse of the conditional Fisher information raised to a power. Common choices for $a$ are $a = 0, 1/2, 1$. 

In this paper, we refer to \cite{blasques2023quasi} and specify $\psi$ as a form more closely related to the SD model:
\begin{equation}\label{psi}
    \psi(y_n,\lambda_n)=S(\lambda_n)\frac{\partial \log q(y_n, \lambda_n)}{\partial \lambda_n},
\end{equation}
where $q(y_n, \lambda_n)$ is a scale family density function but not necessarily $p(y_n|\mathcal{F}_{n-1}; \lambda_n)$. More specifically, $q(y_n, \lambda_n)$ is a hypothetical conditional density of $y_n$, where $\varepsilon_n|\mathcal{F}_{n-1}$ in \eqref{yeq} follows a probability density $g(\cdot)$ rather than the true density $f(\cdot)$, i.e., 
\[q(y,\lambda)=\frac{1}{\sqrt{\varphi(\lambda)}}g\left(\frac{y}{\sqrt{\varphi(\lambda)}}\right).\]
We adopt the notation $\nabla_n$ for the score used in the SD model, and similarly denote $\frac{\partial \log q(y_n, \lambda_n)}{\partial \lambda_n}$ as $\nablaq_n$, called {\it quasi score}. 

The QSD volatility model studied in this paper is constructed by combining \eqref{yeq}, \eqref{peq}, \eqref{psi}, and using $X_n$ to represent the log price as in \eqref{garch}:
\begin{equation}\label{QSDmodel}
    \left\{\begin{aligned}
        &X_n=X_{n-1}+c\lambda_n+y_n,\\
        &\lambda_{n+1}=\omega+\beta\lambda_{n}+\alpha S(\lambda_n) \nablaq_n,
    \end{aligned}
    \right.
\end{equation}
where $y_n=\sqrt{\varphi(\lambda_n)}\varepsilon_n$ is the innovation of $X_n$, and $\varepsilon_n|\mathcal{F}_{n-1}\stackrel{d}{\sim}f(\cdot)$. Note that in \eqref{QSDmodel}, conditioned on $\lambda_n$, $\lambda_{n+1}$ is actually an $\mathcal{F}_n$-measurable random variable. Therefore, we define $Z_n=(X_n,\lambda_{n+1})$, and obtain a new $\sigma$-algebra $\mathcal{F}_{n}=\sigma(Z_n,Z_{n-1},\ldots, Z_0, \lambda_0)$. 

To study the continuous-time limit of $\{Z_n\}_{n \in \mathbb{N}}$, we begin by associating the discrete time indices in $\mathbb{N}$ with a time interval $h$, thereby defining the process $\{Z_{kh}^{(h)}\}_{k \in \mathbb{N}}$ which takes values at times $\{0, h, 2h, \ldots\}$. Next, by connecting these values in a stepwise manner, we construct a càdlàg process $\{Z^{(h)}_t\}_{t\geq 0}$, which is a random element taking values in $D$-space equipped with the Skorokhod topology. Finally, we consider the weak convergence of $\{Z^{(h)}_t\}$ as $h\to0$. The detailed steps of this procedure are outlined in the \ref{appendix1}.

Therefore, we first associate the QSD volatility model \eqref{QSDmodel} to the time interval $h$, resulting in a two-dimensional discrete-time Markov process $Z_{kh}^{(h)}=(X_{kh}^{(h)},\lambda_{(k+1)h}^{(h)})$, with $\sigma$-algebra $\mathcal{F}_{kh}=\sigma(Z^{(h)}_{kh}, Z^{(h)}_{(k-1)h},\ldots, Z^{(h)}_0, \lambda_0^{(h)})$. It follows that  
\begin{equation}\label{discrete-QSD}
    \left\{\begin{aligned}
        &X_{kh}^{(h)}=X_{(k-1)h}^{(h)}+c_h \lambda_{kh}^{(h)}+y_{kh}^{(h)},\\
        &\lambda_{(k+1)h}^{(h)}=\omega_{h}+\beta_{h}\lambda_{kh}^{(h)}+\alpha_{h}S(\lambda_{kh}^{(h)})\nablaq_{kh}^{(h)}.
    \end{aligned}
    \right.
\end{equation}
Where, 
\[y_{kh}^{(h)}=\sqrt{\varphi(\lambda_{kh}^{(h)})}\varepsilon_{kh}^{(h)}, \quad  h^{-1/2}\varepsilon_{kh}^{(h)}\big|\mathcal{F}_{(k-1)h}\stackrel{d}{\sim}f(\cdot),\]
and the quasi score
\[\nablaq_{kh}^{(h)}=\frac{\partial \log q(y_{kh}^{(h)}, \lambda_{kh}^{(h)})}{\partial \lambda_{kh}^{(h)}},\quad q(y,\lambda)=\frac{1}{\sqrt{\varphi(\lambda)h}}g\left(\frac{y}{\sqrt{\varphi(\lambda)h}}\right).\]
Recall that in our setting, $f$ is a probability density function with zero mean, while $g$ is another probability density function, which may differ from $f$. 

Following the same procedure detailed in the \ref{appendix1}, we obtain its continuous-time process $Z_t^{(h)}=(X_t^{(h)}, \lambda_t^{(h)})$ based on $Z_{kh}^{(h)}$, and denote $\mathcal{F}_{t}^{(h)}$ as its generated $\sigma$-algebra. 

\section{Main results}\label{Main results}
In this section, we derive the weak convergence limit of $Z_t^{(h)}$ as $h\to0$. The proof of the main result consists in an application of a general functional central limit theorem presented in \cite{stroock1997multidimensional}. Here, we utilize a simpler, although somewhat less general, version proposed in \cite{nelson1990arch}, which is retailed in the \ref{appendix2}. 

According to the Conditions \ref{con1} of Theorem \ref{corethm}, we need to examine the ratio of certain conditional moments of the process increments to $h$ as $h\to0$, and the most important term in the process is the quasi score $\nablaq$. We therefore begin by proving the following Lemma \ref{nablalem}.
\begin{lemma}\label{nablalem}
    For every $t\geq 0, l\in\mb{N}$, as $h\to0$, the moments $\mb{E}[(\nablaq_{t+h}^{(h)})^l|\mathcal{F}_{t}^{(h)}]=O(1)$, while $\mb{E}[\nablaq_{t+h}^{(h)}\varepsilon_{t+h}^{(h)}|\mathcal{F}_{t}^{(h)}]=O(h^{1/2})$.
\end{lemma}
\begin{proof}
    First, since the condition distributions of $\{\varepsilon_{kh}^{(h)}\}_{k\in \mb{N}}$ are independent of time, the two expectations above are independent of $t$. For simplicity, we omit superscripts and subscripts in the following statements\footnote{i.e., $y$ denotes $y_{t+h}^{(h)}$, $\nablaq$ denotes $\nablaq_{t+h}^{(h)}$, $\lambda$ denotes $\lambda_{t+h}^{(h)}$, $\varepsilon$ denotes $\varepsilon_{t+h}^{(h)}$, $\mathcal{F}$ denotes $\mathcal{F}_{t}^{(h)}$.}. 
    \[\begin{split}
        \nablaq&=\frac{\partial \log q(y, \lambda)}{\partial \lambda}=\frac{\partial}{\partial \lambda}\left[\log \frac{1}{\sqrt{\varphi(\lambda)h}}g\left(\frac{y}{\sqrt{\varphi(\lambda)h}}\right)\right]\\
        &=\frac{-\varphi'(\lambda)}{2\varphi(\lambda)}\left[1+\frac{g'\left(\frac{y}{\sqrt{\varphi(\lambda)h}}\right)}{g\left(\frac{y}{\sqrt{\varphi(\lambda)h}}\right)}\frac{y}{\sqrt{\varphi(\lambda)h}}\right].
    \end{split}\]
    Let $u=\frac{y}{\sqrt{\varphi(\lambda)h}}$, then, 
    \begin{equation}\label{moment}
        \mb{E}\left[\nablaq ^l|\mathcal{F}\right]=\left[\frac{-\varphi'(\lambda)}{2\varphi(\lambda)}\right]^l\int_{-\infty}^{\infty}\left[1+\frac{g'\left(u\right)}{g\left(u\right)}u\right]^l f\left(u\right)\dd u.
    \end{equation}
    Since $y_{t}^{(h)}=\sqrt{\varphi(\lambda_{t}^{(h)})}\varepsilon_{t}^{(h)}$, then $\varepsilon=\frac{y}{\sqrt{\varphi(\lambda)}}=\sqrt{h}u$, thus
    \begin{equation}\label{cmom}
        \mb{E}\left[\nablaq\varepsilon|\mathcal{F}\right]=\frac{-\varphi'(\lambda)}{2\varphi(\lambda)}\sqrt{h}\int_{-\infty}^{\infty}\left[1+\frac{g'\left(u\right)}{g\left(u\right)}u\right]u f\left(u\right)\dd u.
    \end{equation}
    Clearly, as $h\to0$, \eqref{moment} are $O(1)$ and \eqref{cmom} are $O(h^{1/2})$. 
\end{proof}
\begin{remark}
    It is evident from \eqref{moment} that the absolute moments $\mb{E}[|\nablaq|^l|\mathcal{F}]$ are also $O(1)$, thus $l$ can be extended to $\mb{R}^+$ when considering the absolute moments.
\end{remark}
\begin{remark}\label{remark_form}
    The proof also shows that $\mb{E}[\nablaq^l|\mathcal{F}]$ are functions of $\lambda$ and take the form 
    \begin{equation}\label{form}
        \mb{E}[\nablaq^l|\mathcal{F}]=m_l\left[\varphi'(\lambda)/\varphi(\lambda)\right]^l, 
    \end{equation}
    where 
    \begin{equation}\label{m}
        m_l=\frac{1}{(-2)^l}\int_{-\infty}^{\infty}\left[1+\frac{g'(u)}{g(u)}u\right]^l f(u)\dd u. 
    \end{equation}
    which is an $l$-related constant, if it exists. 
\end{remark}

In the case of $g=f$, the integral
\[\int_{-\infty}^{\infty}\left[1+\frac{f'(u)}{f(u)}u\right] f(u)\dd u=\int_{ -\infty}^{\infty}f(u)+uf'(u)\dd u=1+uf(u)\big|_{-\infty}^{\infty}-1=0, 
\]
so that $\mb{E}[\nabla|\mathcal{F}]=0$ in the SD model, but it may not hold for many cases where $g\neq f$. On the other hand, the score can be regarded as the innovation of the time-varying parameter, which should be ``fair'' in a certain sense. Therefore, we first consider the special case where $\mb{E}[\nablaq|\mathcal{F}]=0$, as the results in this case are relatively concise and the proof procedure is fundamental. Subsequently, we address the general case where $\mb{E}[\nablaq|\mathcal{F}]\neq 0$. 

Before commencing our analysis, we require the following assumptions to guarantee that $Z_t^{(h)}$ converges to a nondegenerate diffusion. 

\begin{assumption}[Rate of scaling]\label{a1}
    There exist $c, \omega, \theta \in\mb{R}$, $\alpha\in \mb{R}\backslash\{0\}$ such that $h^{-1}c_h\to c$, $h^{-1}\omega_h\to\omega$, $h^{-1}(1-\beta_h)\to\theta$, $h^{-1/2}\alpha_h\to\alpha$ as $h\to0$.
\end{assumption}
\begin{assumption}[Existence of some moments]\label{a2}
    Let $U$ be a random variable with probability density $f$. We assume that 
    \[\mb{E}[U^2g'(U)/g(U)]=\rho<\infty, \]
    and there exists a $\delta>0$ such that 
    \[\mb{E}|U|^{2+\delta}<\infty, \quad \mb{E}[|\nablaq|^{2+\delta}|\mathcal{F}]<\infty.\]
    This implies that the second moment are finite and we assume $\mb{E}(U^2)=\eta>0$, $\mb{E}[(\nablaq)^2|\mathcal{F}]=\gamma(\lambda)>0$.
\end{assumption}
As will be demonstrated in the subsequent analysis, $\rho$ plays a pivotal role. At this point, we would like to highlight two specific cases where $\rho$ vanishes. Firstly, when $g = f$, which corresponds to the case of the SD model, we have
\[\rho=\mb{E}[U^2f'(U)/f(U)]=\int_{-\infty}^{\infty}u^2f'(u)\dd u=u^2f\big|_{-\infty}^{\infty}-2\int_{-\infty}^{\infty}uf(u)\dd u=0.\]
Moreover, under the condition that both $f$ and $g$ are symmetric distributions, then it follows that $u^2 g'(u)f(u)/g(u)$ is an odd function, and consequently, $\rho = 0$.

\begin{assumption}[Continuity]\label{a3}
    There exists a compact set $\Gamma\subset\mb{R}$ such that $S(\lambda)\varphi'(\lambda)/\varphi(\lambda)$ is continuous on $\Gamma$. 
\end{assumption}

Note that we suppose $\eta$, $\alpha$, $\gamma(\lambda)$ are all non-zero in assumptions. Because $\eta=0$ implies that the variance of the innovation is zero, and $\alpha=0$ or $\gamma(\lambda)=0$, as we will see later, leads to a degenerate SDE, both cases are trivial. In fact, according to representation \eqref{form}, $\gamma(\lambda)>0$ is equivalent to the link function $\varphi(\cdot)$ being strictly monotonic. This condition is satisfied by common choices such as $\varphi(x)=x$ or $\varphi(x)=\ee^{x}$.

\subsection{Quasi score with zero mean}
In this part, we assume that $f$ and $g$ are two probability density functions such that $\mb{E}[\nablaq|\mathcal{F}]=0$, i.e. the integral $\int_{-\infty}^{\infty}[1+\frac{g'(u)}{g(u)}u] f(u)\dd u=0$. 

\begin{theorem}\label{main thm}
    Under assumptions \ref{a1},\ref{a2},\ref{a3} and $\mb{E}[\nablaq|\mathcal{F}]=0$. If $4\eta m_2-\rho^2>0$ and there exists a random variable $Z_0$ with probability measure $\nu_0$ such that $Z_0^{(h)}\tod Z_0$ as $h\to0$, then $Z_t^{(h)}=\left(X_t^{(h)}, \lambda_{t+h}^{(h)}\right)$ weakly converges to the following It\^{o} SDE's unique weak solution $Z_t=(X_t,\lambda_t)\in \mb{R}\times \Gamma$ as $h\to0$:
    \begin{equation}\label{QSD-sde}
        \left\{\begin{aligned}
            &\dd X_t=c\lambda_t\dd t+\sqrt{\varphi(\lambda_t)\eta}\dd W_t^{(1)},\\
            &\dd \lambda_t=(\omega-\theta\lambda_t)\dd t+\alpha S(\lambda_t)\sqrt{\gamma(\lambda_t)}\dd W_t^{(2)},\\
            &\mbox{\rm Cov}\left[\dd W_t^{(1)},\dd W_t^{(2)}\right]=-\frac{\rho}{2\sqrt{\eta m_2}}\dd t,\\
            &\mb{P}(Z_0\in B)=\nu_0(B),  \quad \mbox{for any} \ B\in\mathcal{B}(\mb{R}\times \Gamma).
        \end{aligned}\right.
    \end{equation}
    Where $W_t^{(1)}, W_t^{(2)}$ are two (potentially) correlated standard Brownian motions.
\end{theorem}

\begin{proof}
    See Appendix \ref{appendixB}. 
\end{proof}

\begin{remark}
    In the theorem, aside from the non-zero constraint imposed on $\alpha$, no restrictions are placed on $\omega$ and $\theta$. Indeed, the convergence result holds for arbitrary $\omega$ and $\theta$. However, as a volatility process, the limiting SDE requires $\omega, \theta>0$ from an economic and practical perspective to ensure that it is mean-reverting and has a positive long-term mean.
\end{remark}

In our proof, there are two distinctions from the approaches of \cite{buccheri2021continuous}. First, our assumptions are more relaxed. Specifically, we do not require the fourth moment of the innovation and score to be finite, but only need the existence of $\delta$ such that the $2+\delta$-th moment is finite. In their proof, by directly computing 
\[h^{-1}\mb{E}[(\lambda_{t+h}^{(h)}-\lambda_{t}^{(h)})^4|\mathcal{F}_{t}^{(h)}]=h\alpha^4 S(\lambda_{t+h}^{(h)})^4\mb{E}[(\nabla_{t+h}^{(h)})^4|\mathcal{F}_t^{(h)}]+O(h^\gamma), \ \gamma\geq 3/2,\]
the finiteness of the fourth moment ensures the existence of $\delta=2$ such that $c_{h,i,\delta}(x,t)$ tends to 0 uniformly. However, in fact, by applying Jensen's inequality to bound
\[h^{-1}\mb{E}[|\lambda_{t+h}^{(h)}-\lambda_{t}^{(h)}|^{2+\delta}|\mathcal{F}_{t}^{(h)}]\lesssim \frac{3^{1+\delta}}{h}\left|\alpha_{h}S(\lambda_{t+h}^{(h)})\right|^{2+\delta}\mb{E}\left(\left|\nablaq_{t+h}^{(h)}\right|^{2+\delta}|\mathcal{F}_t^{(h)}\right),\]
the existence of the $2+\delta$-th moment is enough (For details, please refer to Appendix \ref{appendixB}). Second, in \cite{buccheri2021continuous}, the existence and uniqueness of the limiting SDE are discussed only for two specific link functions, whereas in this paper, we provide a direct proof within the theorem itself.

Next, we consider a more general case. Note that when  calculating the drift term of $\lambda_t$ in the proof , the key to the limit
    \[\lim_{h\to0}\frac{\alpha_{h}S(\lambda_{t+h}^{(h)})\mb{E}[\nablaq_{t+h}^{(h)}|\mathcal{F}_{t}^{(h)}]}{h}=0\]
rests on the assumption that $\mb{E}[\nablaq_{t+h}^{(h)}|\mathcal{F}_{t}^{(h)}]=0$. Otherwise, it will tend to infinity due to $\alpha_h=O(h^{1/2})$. Therefore, a new scaling form of the QSD volatility model is required when the quasi-score has a non-zero mean. 

\subsection{Quasi score with non-zero mean}
According to Remark \ref{remark_form}, the conditional mean of quasi score has the form $m_1\varphi'(\lambda)/\varphi(\lambda)$. Thus in this part, we consider the case where $m_1=\mu\neq 0$. To address such cases, we can extract the mean value from $\nablaq$ and incorporate it into the drift term. Specifically, we rewrite the updating equation for $\lambda_n$ as follows:
\[\begin{split}
    \lambda_{n+1}&=\omega+\beta\lambda_{n}+\alpha S(\lambda_{n})\nablaq_{n}\\
    &=\omega+\frac{\mu\alpha S(\lambda_{n})\varphi'(\lambda_{n})}{\varphi(\lambda_{n})}+\beta\lambda_{n}+\alpha S(\lambda_{n})\left[\nablaq_{n}-\frac{\mu\varphi'(\lambda_{n})}{\varphi(\lambda_{n})}\right]\\
    &=\omega+\frac{\mu\alpha S(\lambda_{n})\varphi'(\lambda_{n})}{\varphi(\lambda_{n})}+\beta\lambda_{n}+\alpha S(\lambda_{n})\widetilde{\nablaq}_{n}.
\end{split}
\]
Relating this to the time interval $h$, we obtain
\begin{equation}\label{new-discrete-QSD}
    \left\{\begin{aligned}
        &X_{kh}^{(h)}=X_{(k-1)h}^{(h)}+c_h \lambda_{kh}^{(h)}+y_{kh}^{(h)},\\
        &\lambda_{(k+1)h}^{(h)}=\omega_{h}+\frac{\mu_h\alpha_h S(\lambda_{kh}^{(h)})\varphi'(\lambda_{kh}^{(h)})}{\varphi(\lambda_{kh}^{(h)})}+\beta_{h}\lambda_{kh}^{(h)}+\alpha_{h}S(\lambda_{kh}^{(h)})\widetilde{\nablaq}_{kh}^{(h)},
    \end{aligned}
    \right.
\end{equation}
and its continuous-time version $Z_t^{(h)}=(X_t^{(h)}, \lambda_{t+h}^{(h)})$ with $\sigma$-algebra $\mathcal{F}_t^{(h)}$. 
\begin{theorem}\label{main thm2}
    Under assumptions \ref{a1},\ref{a2},\ref{a3} and we further assume $h^{-1/2}\mu_h\to \mu$, $4\eta (m_2-\mu^2)-\rho^2>0$ and there exists a random variable $Z_0$ with probability measure $\nu_0$ such that $Z_0^{(h)}\tod Z_0$ as $h\to0$, then $Z_t^{(h)}=\left(X_t^{(h)}, \lambda_{t+h}^{(h)}\right)$ weakly converges to the following It\^{o} SDE's unique weak solution $Z_t=(X_t,\lambda_t)\in \mb{R}\times \Gamma$ as $h\to0$:
    \begin{equation}\label{new-QSD-sde}
        \left\{\begin{aligned}
            &\dd X_t=c\lambda_t\dd t+\sqrt{\varphi(\lambda_t)\eta}\dd W_t^{(1)},\\
            &\dd \lambda_t=\left[\omega-\theta\lambda_t+\frac{\alpha\mu S(\lambda_t)\varphi'(\lambda_t)}{\varphi(\lambda_t)}\right]\dd t+\alpha \sqrt{m_2-\mu^2}\frac{S(\lambda_t)\varphi'(\lambda_t)}{\varphi(\lambda_t)}\dd W_t^{(2)},\\
            &\mbox{\rm Cov}\left[\dd W_t^{(1)},\dd W_t^{(2)}\right]=-\frac{\rho}{2\sqrt{\eta (m_2-\mu^2)}}\dd t, \\
            &\mb{P}(Z_0\in B)=\nu_0(B),  \quad \mbox{for any} \ B\in\mathcal{B}(\mb{R}\times \Gamma).
        \end{aligned}\right.
    \end{equation}
    Where $W_t^{(1)}, W_t^{(2)}$ are two (potentially) correlated standard Brownian motions.
\end{theorem}

\begin{proof}
    See Appendix \ref{appendixC}. 
\end{proof}

Comparing the continuous-time limits \eqref{QSD-sde} and \eqref{new-QSD-sde}, we find that \eqref{QSD-sde} corresponds exactly to the special case of $\mu=0$ in \eqref{new-QSD-sde}. From the results of the two theorems, we can find that the most significant distinction between the continuous-time limit of the QSD and the SD model manifests in the correlated nature of the two Brownian motions. We now proceed to analyze this phenomenon. 

\subsection{QSD model as a general framework for correlation emergence}
In this part, we analyze the correlation coefficient between the two Brownian motions in the continuous-time limit and it reveals that the QSD model provides a general framework for generating such correlation. 

It is evident that in both theorems, $\rho$ serves as a crucial quantity governing the correlation coefficient, i.e.
\begin{equation}\label{formrho}
    \mbox{\rm Cov}\left[\dd W_t^{(1)},\dd W_t^{(2)}\right]=-\frac{\rho}{2\sqrt{\eta (m_2-\mu^2)}}\dd t.
\end{equation}  
We use the word ``potentially'' in the statements because there are two cases, as previously discussed under Assumption \ref{a2}, will cause $\rho=0$, which indicates that the two Brownian motions are independent. Specifically, when $g=f$, the continuous-time limit recovers the result of \cite{buccheri2021continuous} about SD volatility models. Furthermore, when $g$ is the density of the standard normal distribution and $f$ is symmetric, it recovers the result of \cite{nelson1990arch} about general GARCH models. We conclude it as the following corollary: 
\begin{corollary}\label{coro}
    The two Brownian motions in the continuous-time limit of QSD volatility models can be correlated only if the distributions driving the innovation and computing the score are different, and at least one of which is asymmetric. 
\end{corollary}

It is worth noting that in the limit of GARCH or more generally SD volatility models, the transition from a single to two independent sources of randomness is quite surprising. In fact, the above corollary illustrates that this phenomenon can arise not only within SD framework but also due to the symmetry of the distribution, with GARCH model employing Gaussian innovations lying at the intersection of both. The QSD volatility model provides a bridge for this transition. Specifically, by adjusting $f$ and $g$, the correlation between these two sources of randomness can be modulated, and it is even possible to revert to a single source of randomness when $\rho=1$.

Indeed, the emergence of correlated Brownian motions in the continuous-time limit does not constitute a novel discovery unique to this paper. \cite{trifi2006issues} has previously analyzed the general GARCH-M model where the innovation term follows an arbitrary $D(0,1)$ distribution, demonstrating that the correlation coefficient is determined by the skewness of the innovation. Furthermore, certain models incorporating asymmetric terms into the time-varying parameters equation and deriving the volatility process through specific mappings, such as EGARCH \citep{nelson1991conditional} and CEV-ARCH \citep{fornari2006approximating}. These models are also capable of generating correlated Brownian motions \citep{nelson1990arch,trifi2006issues}. However, it should be pointed out that they can also be encompassed within the framework of the QSD model for comprehensive analysis. In the following analysis, we elucidate this point. To maintain generality, we do not specify that $f\sim N(0,1)$, but only require that it belongs to the class of distributions which has zero mean.

\subsubsection{GARCH-M model}

When $\varphi$ is identity mapping and $g$ is restricted to the density of the standard normal distribution, we obtain the general GARCH-M model. In this context, 
\[\rho=\mb{E}_{f}[U^2g'(U)/g(U)]=-\mb{E}_f[U^3],\]
which recovers the result of \cite{trifi2006issues} and this quantity obviously vanishes when $f$ is symmetric. 

\subsubsection{EGARCH model}

In the EGARCH model, the dynamic of $\ln \sigma_t^2$ is characterized, with its asymmetric term takes the form
\begin{equation}\label{EGARCH}
    \psi(\varepsilon)=\theta \varepsilon+\gamma [|\varepsilon|-\mb{E}|\varepsilon|].
\end{equation}

Within the framework of the QSD model, we can obtain the EGARCH model by selecting $\varphi(x)=\ee^x$ and 
\[g(x)=\frac{(\theta+\gamma)^2}{2}|x|\ee^{-\theta x-\gamma|x|}.\]
It can be verified that, after proper integration of the constants, \eqref{EGARCH} corresponds to the standardized score function of density $q(y,\lambda)=\ee^{-\lambda/2}g(y\ee^{-\lambda/2})$. Note that $g(x)$ is asymmetric except when $\theta = 0$, and this asymmetry gives rise to the emergence of correlation. Specifically, 
\[\rho=\mb{E}_f[U^2(\theta+\text{sgn}(U)\gamma+U^{-1})]=\theta+\gamma\mb{E}_f[\text{sgn}(U)U^2].\]
Furthermore, if $f$ is symmetric, then $\rho=\theta$. 

In fact, through direct observation of \eqref{EGARCH}, it can be seen that $\psi$ is precisely a linear combination of $\varepsilon$ and $|\varepsilon|-\mb{E}|\varepsilon|$, and these two components are uncorrelated when $f$ is symmetric. By Donsker's invariance principle, they weakly converge to a two-dimensional Brownian motion. Therefore, the perturbation of the limiting process is essentially synthesized from two independent Brownian motions, with the correlation determined by the coefficients $\theta$. When $\theta=0$, the process is entirely driven by the second independent Brownian motion.

\subsubsection{CEV-ARCH model}

The CEV-ARCH model is capable of approximating any CEV-diffusion model for stochastic volatility. Referring to Eq. (14) in \citet{fornari2006approximating}, the dynamic related to volatility is given by 
\[\sigma_{n+1}^{\delta}=\omega+\beta \sigma_{n}^{\delta}+\alpha |\varepsilon_n|^{\delta}(1-\gamma\text{sgn}(\varepsilon_n))^{\delta}\sigma_{n}^{\delta},\]
with $\delta>0, \gamma\in (-1,1)$. 

Under the assumption that $\varepsilon$ follows a Generalized Error Distribution (GED), \cite{fornari2006approximating} derived an expression for the correlation coefficient $\rho$ of the limiting CEV process, as presented in their Eq. (16). 

Similarly, within the framework of the QSD model, we can obtain the CEV-ARCH model by selecting $\varphi(x)=x^{2/\delta}$ and 
\[g(x)=\frac{(1-\gamma^2)\delta^{1-1/\delta}}{2\Gamma(1/\delta)}\exp\left[-\frac{(1-\gamma\text{sgn}(x))^\delta |x|^\delta}{\delta}\right].\]
Interestingly, while $g(x)$ bears resemblance to the density of GED with parameter $\delta$, it distinguishes itself through its asymmetric structure, where the weights on the positive and negative components are $(1 - \gamma)^\delta$ and $(1 + \gamma)^\delta$, respectively. Therefore, $g$ can be characterized as an asymmetric GED, where $\gamma\in (-1,1)$ is a skewness parameter. With $\gamma = 0$, the distribution reduces to the GED. 

Therefore, the CEV-ARCH model is essentially a QSD volatility model driven by the asymmetric GED\footnote{One may observe that the parameter $\delta$ in $g$ must coincide with that in the link function. However, in the Eq. (21) of \cite{fornari2006approximating}, they proposed a more generalized form, where the asymmetric term's parameter becomes $\delta\eta$. This modification liberates the parameter of the asymmetric GED, which can be adjusted through $\eta$.}. This asymmetry naturally accounts for the emergence of correlation in the CEV-ARCH model. Specifically, 
\[\rho=\frac{1}{\delta}\mb{E}_f[\text{sgn}(U)(1-\gamma\text{sgn}(U))^\delta |U|^{\delta+1}].\]
In the case of $\gamma=0$, $\rho=\delta^{-1}\mb{E}_f[\text{sgn}(U)|U|^{\delta+1}]$. Obviously, it vanishes when $f$ is symmetric. 

From the preceding analysis and main theorems, we can observe that the correlation is actually independent of the link function but solely determined by the distributions $f$ and $g$. Consequently, for models such as EGARCH, the emergence of correlation is not attributed to the transformation of volatility but rather to the incorporated asymmetric term, which essentially stems from the asymmetric density function $g$. 

Although Corollary \ref{coro} is necessary, it is nearly sufficient as well. Because, in the context of distinct and asymmetric density functions $f$ and $g$, constructing cases where the integral $\int_{-\infty}^{\infty}u^2 g'(u)f(u)/g(u)\dd u=0$ is nontrivial. Therefore, it is imperative to point out that the results of Corollary \ref{coro} are insightful and general, which is to some extent attributed to the generality of the QSD model.

\section{Two examples based on symmetric and asymmetric distributions}\label{Two examples}

In this section, we consider two examples of QSD volatility model: one of which is driven by Student’s t-distribution (thus encompassing the normal as a limiting case), and another is driven by skewed Student’s t-distribution of \cite{azzalini2003distributions}. As Corollary \ref{coro} indicates, correlation arises only in asymmetric cases. Thus, we consider the diffusion coefficient of continuous-time limit for the former and the correlation coefficient for the latter.

Here, we focus on the case where $\varphi(x)=x$, implying that the conditional variance $\sigma_t^2$ is treated as the time-varying parameter $\lambda_t$. Alternatively, one can also easily set $\varphi(x)=\ee^{x}$ to model the dynamics of $\ln \sigma_t^2$ or $\varphi(x)=x^{2/\delta}$ to model, just like in the EGARCH-type models. 

\subsection{The QSD-T model}

When considering the QSD volatility model based on two standardized Student's t-distribution with different degrees of freedom $v_1, v_2$, we obtain a model called QSD-T,
\begin{equation}\label{QSD-T}
    \begin{split}
        &y_n=\sigma_n\varepsilon_n, \quad \varepsilon_n\stackrel{i.i.d.}{\sim}t_{v_1}\\
        &\sigma^2_{n+1}=\omega+\beta \sigma^2_{n}+\alpha \sigma^2_{n}\left[\frac{(v_2+1)y_n^2}{(v_2-2)\sigma_n^2+y_n^2}-1\right].
    \end{split}
\end{equation}
This model encompasses many common models, such as GARCH, t-GARCH, Beta-t \citep{harvey2008beta}, Beta-normal \citep{banulescu2018volatility}. We list these models in the Table \ref{five models}, where the standard normal distribution case corresponds to $v_1\to\infty$ or $v_2\to\infty$. Recall that we use the density $f$ to drive $y_n$ and use the density $g$ to compute $\nablaq$ in the context of Section \ref{QSD volatility models}. 
\begin{table}[htbp]
    \centering
    \caption{Five QSD models based on Student's t-distributions}\label{five models}
    \begin{tabular}{llll}
    \toprule
    Model       & $f\sim$         &  $g\sim$   & $\nablaq_n$ \\
    \midrule
    GARCH       & $N(0,1)$          & $N(0,1)$          & $\frac{1}{2\sigma_n^2}(\frac{y_n^2}{\sigma_n^2}-1)$          \\
    t-GARCH     & $t_v$        & $N(0,1)$          & $\frac{1}{2\sigma_n^2}(\frac{y_n^2}{\sigma_n^2}-1)$        \\
    Beta-t      & $t_v$           & $t_v$           &  $\frac{1}{2\sigma^2_{n}}\left[\frac{(v+1)y_n^2}{(v-2)\sigma_n^2+y_n^2}-1\right]$         \\
    Beta-normal & $N(0,1)$          & $t_v$           &  $\frac{1}{2\sigma^2_{n}}\left[\frac{(v+1)y_n^2}{(v-2)\sigma_n^2+y_n^2}-1\right]$         \\
    QSD-T       & $t_{v_1}$ & $t_{v_2}$ &   $\frac{1}{2\sigma^2_{n}}\left[\frac{(v_2+1)y_n^2}{(v_2-2)\sigma_n^2+y_n^2}-1\right]$        \\
    \bottomrule
    \end{tabular}
\end{table}

The scaling function of score $S(\sigma_n^2)=2\sigma_n^4$, which is proportional to the inverse of the Fisher information \citep{creal2013generalized, buccheri2021continuous,blasques2023quasi}. As discussed in Section \ref{QSD volatility models}, there are various options for selecting the scaling function. However, our choice is motivated by two considerations: (1) when $f,g\sim N(0,1)$, $S(\sigma_n^2)\nablaq_n$ precisely corresponds to the updating scheme of the GARCH model; and (2) employing a unified scaling function for the models in Table 1 facilitates meaningful comparisons, as they all belong to the QSD-T models. 

For a more detailed exposition of QSD-T model, we refer readers to \cite{blasques2023quasi}, which provides a comprehensive treatment of statistical inference, including parameter estimation and hypothesis testing for its potential reduction to beta-t, t-GARCH, and GARCH specifications. The study conducted an empirical analysis on 400 US stocks from the S\&P 500, revealing that approximately 90\% of the stocks rejected the null hypothesis of t-GARCH, while over 50\% rejected beta-t. They estimated the reciprocal of the degrees of freedom. While the three models yield consistent results for $1/v_1$, the estimates for $1/v_2$ fall between 0 and $1/v_1$. Therefore, the News impact curve of the QSD-T model lie between that of the other two models (t-GARCH, beta-t) and that there is more heterogeneity.

According to \eqref{new-QSD-sde}, the continuous-time limit of \eqref{QSD-T} is\footnote{For brevity, we set $c=0$, i.e. $X_n=\sum_{i=1}^n y_i$. }
\begin{equation}\label{coQSD-T}
    \left\{\begin{aligned}
        &\dd X_t=\sigma_t\dd W_t^{(1)},\\
        &\dd \sigma_t^2=\left[\omega-(\theta-2\alpha\mu)\sigma_t^2\right]\dd t+2\alpha \sqrt{m_2-\mu^2}\sigma_t^2\dd W_t^{(2)},\\
        &\mb{P}(Z_0\in B)=\nu_0(B),  \quad \mbox{for any} \ B\in\mathcal{B}(\mb{R}\times \Gamma).
    \end{aligned}\right.
\end{equation}
where $W^{(1)}, W^{(2)}$ are two independent Brownian motions due to the symmetry of $f$ and $g$. We now focus on diffusion coefficient of $\sigma_t^2$. Both GARCH and Beta-t belong to SD models so that $\mu=0$, it can also easily compute $\mu=0$ for t-GARCH model. Thus we have $m_2=1/2$ in GARCH model, $m_2=\frac{v-1}{2(v-4)}$ in t-GARCH model, and $m_2=\frac{v}{2(v+3)}$ in Beta-t model. It recovers the results of \cite{buccheri2021continuous}. In the last two models, $\mu\neq 0$ in general, and we can compute $m_2-\mu^2$ numerically. We compare the $2\sqrt{m_2-\mu^2}$ of both models when the degrees of freedom change, see Figure \ref{beta}. In fact, we can recover the Beta-normal ($v_1\to\infty$), t-GARCH ($v_2\to\infty$), GARCH ($v_1,v_2\to\infty$), Beta-t ($v_1=v_2$) models from the case of QSD-T models in Figure \ref{qsdt} . 
\begin{figure}[htbp]
        \center
        \subfigure[Beta-t and Beta-normal models.]{
            \centering
            \includegraphics[scale=0.5]{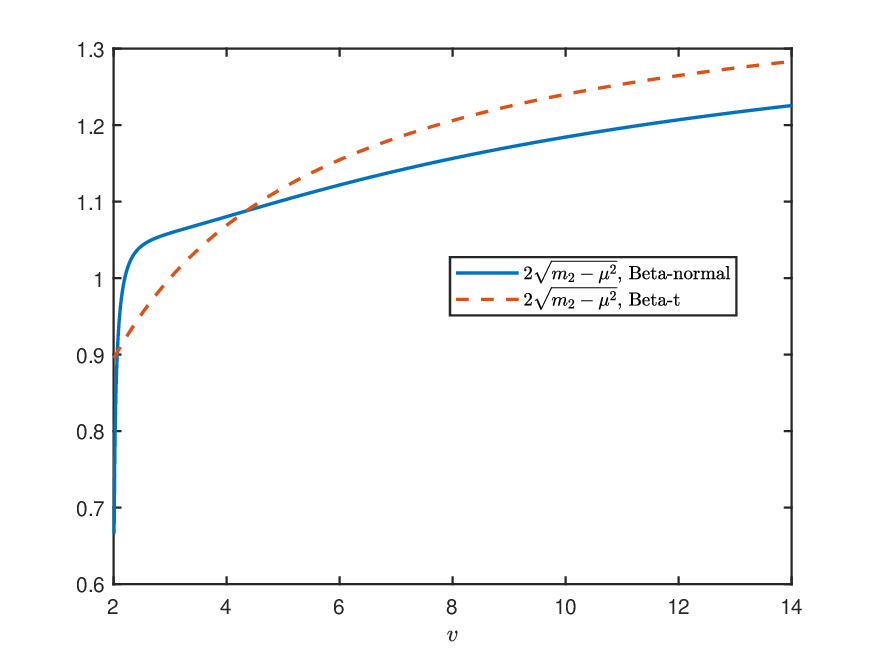}
            \label{betat}
        }\hspace{-10pt}
        \subfigure[QSD-T models.]{
            \centering
            \includegraphics[scale=0.59]{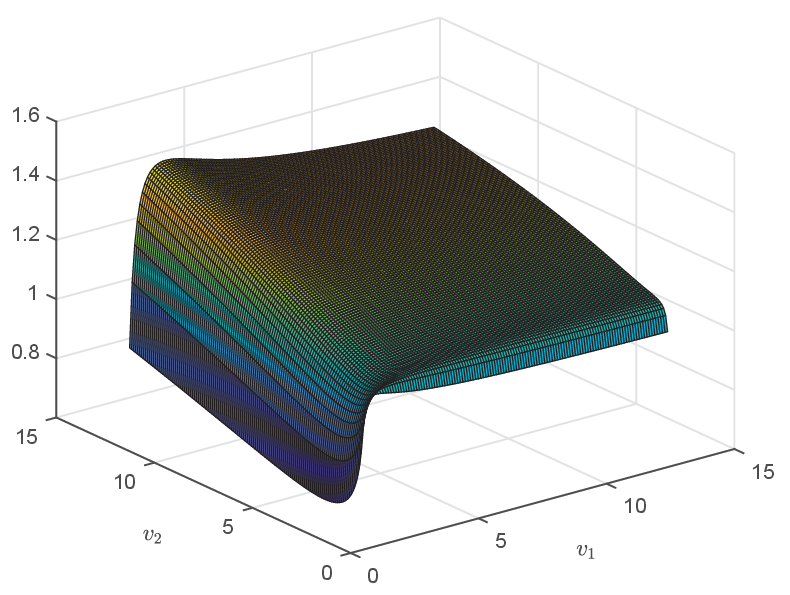}
            \label{qsdt}
        }
        \caption{The diffusion coefficients of Beta-t, Beta-normal and QSD-T models in different degrees of freedom. \label{beta}}
\end{figure}
It can be found that there exists $k\approx 4$ such that when $v>k$, the volatility of volatility of the Beta-normal diffusion is smaller than that of the Beta-t, it is predictable because the tail of normal density is lighter than t. When $v$ is small, the volatility of volatility of the Beta-normal diffusion increases rapidly with $v$, surpassing that of Beta-t, but remains finite.

\subsection{QSD-ST model}
When replacing the t-distribution to skew t-distribution in QSD-T model, we obtain the QST-ST model. The density function of skew t-distribution are given by
\[
f(x ; \varrho, v)=\left\{\begin{array}{l}
K(v)\left[1+\frac{1}{v}\left(\frac{x}{2 \varrho}\right)^{2}\right]^{-\frac{v+1}{2}}, \quad x \leq 0, \\
K(v)\left[1+\frac{1}{v}\left(\frac{x}{2(1-\varrho)}\right)^{2}\right]^{-\frac{v+1}{2}}, \quad x>0,
\end{array}\right.
\]
where $\varrho\in(0,1)$ is the skewness parameter, $\varrho>1/2$ implies $f$ is left-skewed and vice versa right-skewed, when $\varrho=1/2$, it recovers Student's t-distribution. Let $v>0$ is degree of freedom, and $K(v)=\Gamma((v+1)/2)/[\sqrt{\pi v}\Gamma(v/2)]$. We first centered and standardized it in order to serve as the distribution of $\varepsilon$. According to \cite{zhu2010generalized}, the moments of ST random variable $X$ are given by
\[\mb{E}(X^k)=(2\sqrt{v})^k[(-1)^k\varrho^{k+1}+(1-\varrho)^{k+1}]\frac{\Gamma(\frac{k+1}{2})\Gamma(\frac{v-k}{2})}{\sqrt{\pi}\Gamma(\frac{v}{2})},\]
so that the mean and standard deviation are
\begin{equation}\label{ba}
    b=\frac{2\Gamma(\frac{v-1}{2})}{\sqrt{\pi}\Gamma(\frac{v}{2})}\sqrt{v}(1-2\varrho), \quad a=\left[\frac{4v}{v-2}(3\varrho^2-3\varrho+1)-b^2\right]^{1/2}.
\end{equation}
Therefore, the density of ST distribution $t_{v,\varrho}$ we utilize is
\[
f(x)=\left\{\begin{array}{l}
a K\left(v\right)\left[1+\frac{1}{v}\left(\frac{a x+b}{2 \varrho}\right)^{2}\right]^{-\frac{v+1}{2}}, \quad x \leq -b/a \\
a K\left(v\right)\left[1+\frac{1}{v}\left(\frac{a x+b}{2\left(1-\varrho\right)}\right)^{2}\right]^{-\frac{v+1}{2}}, \quad x>-b/a
\end{array}\right..
\]

In QSD-ST model, we suppose that $f\sim t_{v_1, \varrho_1}$ and $g\sim t_{v_2, \varrho_2}$. We denote $b_i, a_i$ as the case where $(v, \varrho)$ is replaced by $(v_i, \varrho_i)$ in \eqref{ba}, and let $\varrho^*_i=\varrho_i 1_{\{x\leq -b_i/a_i\}}+(1-\varrho_i) 1_{\{x>-b_i/a_i\}}$. Then we have the following QSD-ST model:
\begin{equation}\label{QSD-ST}
    \begin{split}
        &y_n=\sigma_n\varepsilon_n, \quad \varepsilon_n\stackrel{i.i.d.}{\sim}t_{v_1,\varrho_1}\\
        &\sigma^2_{n+1}=\omega+\beta \sigma^2_{n}+\alpha \sigma^2_{n}\left[\frac{a_2 \varepsilon_n(a_2\varepsilon_n+b_2)(v_2+1)}{(a_2 \varepsilon_n+b_2)^2+4v_2{\varrho^*_2}^2}-1\right].
    \end{split}
\end{equation}
The continuous-time limit of this model has the same form as \eqref{coQSD-T}, except that two Brownian motions have a correlation coefficient with $\mbox{\rm Cov}\left[\dd W_t^{(1)},\dd W_t^{(2)}\right]=-\frac{\rho}{2\sqrt{\eta (m_2-\mu^2)}}\dd t.$ We compute it when $v_1, v_2, \varrho_1$, $\varrho_2$ change, see Figures \ref{qsdst_v} and \ref{qsdst_a}. 
\begin{figure}[htbp]
    \center
    \subfigure[Fixed $\varrho_1=2/3$, $v_1$ and $v_2$ change in the case of $\varrho_2=1/3$, $1/2$ and $2/3$.]{
        \centering
        \includegraphics[scale=0.31]{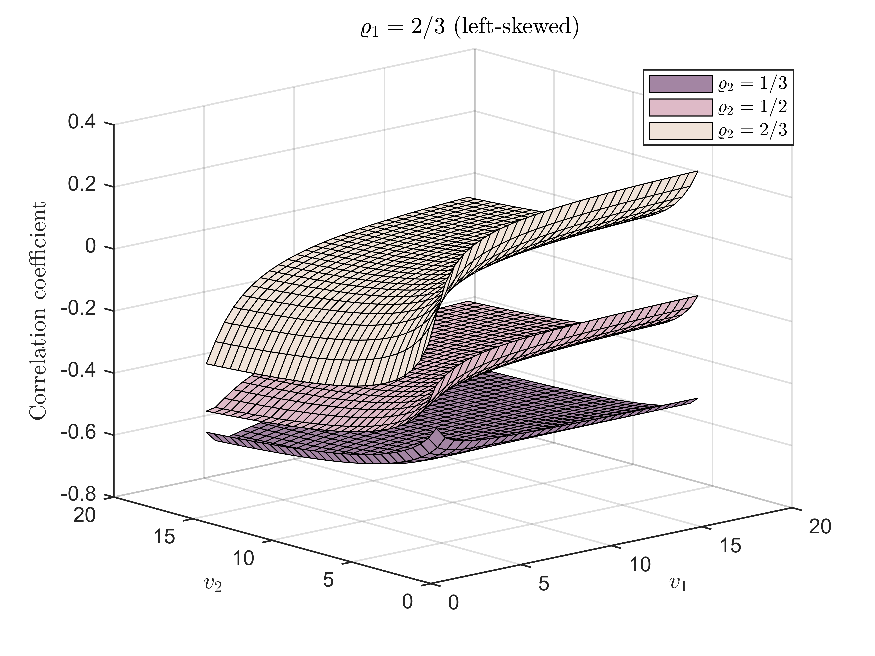}
        \label{qsdstv}
    }\hspace{10pt}
    \subfigure[Fixed $\varrho_1=1/2$, $v_1$ and $v_2$ change in the case of $\varrho_2=1/3$, $1/2$ and $2/3$.]{
        \centering
        \includegraphics[scale=0.31]{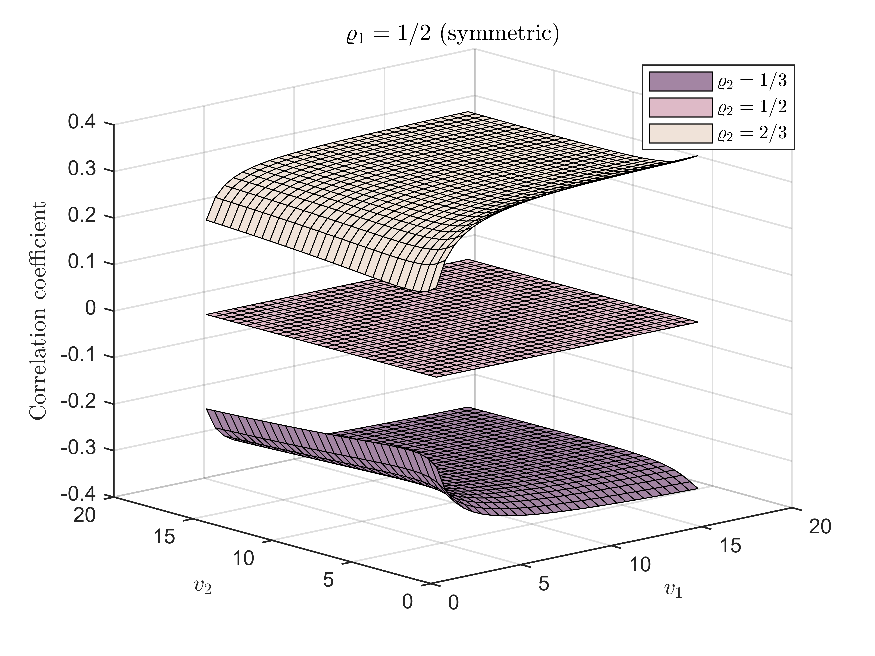}
        \label{qsdstv_sym}
    }
    \hspace{10pt}
    \subfigure[Fixed $\varrho_1=1/3$, $v_1$ and $v_2$ change in the case of $\varrho_2=1/3$, $1/2$ and $2/3$.]{
        \centering
        \includegraphics[scale=0.31]{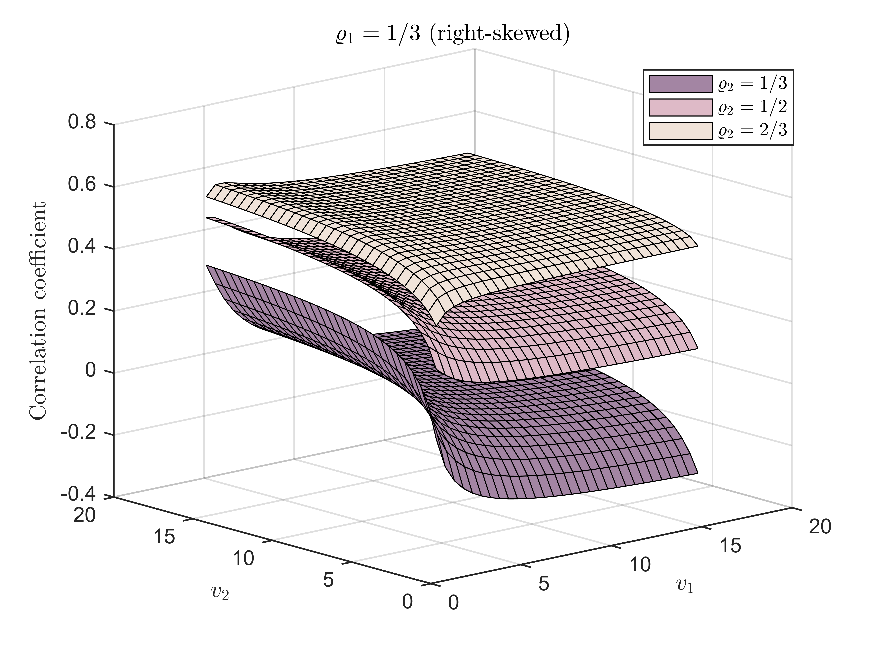}
        \label{qsdstv_right}
    }
    \caption{The correlation coefficient of QSD-ST models: fixed distribution of $\varepsilon_n$ is left-skewed, symmetric or right-skewed. \label{qsdst_v}}
\end{figure}

In Figure \ref{qsdstv}, it can be observed that when $\varepsilon_n$ is left-skewed but $g$ is right-skewed, two Brownian motions exhibit significant negative correlations for all appropriate $v_1$ and $v_2$. So it characterizes the leverage effect, aligning with the advantages of the QSD model discussed in \cite{blasques2023quasi}. Moreover, it can still produce a leverage effect even if $g$ is left-skewed, which would fail in the discrete-time case. Specifically, when $v_2$ is not too small, that is the tail of $g$ is not too heavy, it can also characterize the leverage effect. However, when $\varepsilon_n$ is right-skewed, the right-skewness of $g$ appears to be necessary, and at this point, the tail of $\varepsilon_n$ should not be too heavy, as shown in Figure \ref{qsdstv_right}. Finally, when $\varepsilon_n$ is symmetric as in Figure \ref{qsdstv_sym}, $g$ needs to be right-skewed to characterize the leverage effect, which is consistent with the empirical results of \cite{blasques2023quasi} on 400 US stocks. Additionally, when $v_1=v_2$, it obtain the continuous-time limit of the Beta-st model proposed by \cite{harvey2017volatility} but there is no correlation according to Corollary \ref{coro}. 

We can further discuss the implication of $v$. By comparing the two figures in Figure \ref{qsdst_a}, it can be seen that the heavy-tail of $\varepsilon_n$ can significantly produce a leverage effect even when $\varepsilon_n$ and $g$ are both left-skewed. Conversely, it also exacerbates the inverse leverage effect when $\varepsilon_n$ and $g$ are both right-skewed. 

\begin{figure}[htbp]
    \center
    \subfigure[Fixed $v_1=4$, $\varrho_1$ and $\varrho_2$ change in the case of $v_2=4$, $8$ and $20$.]{
        \centering
        \includegraphics[scale=0.38]{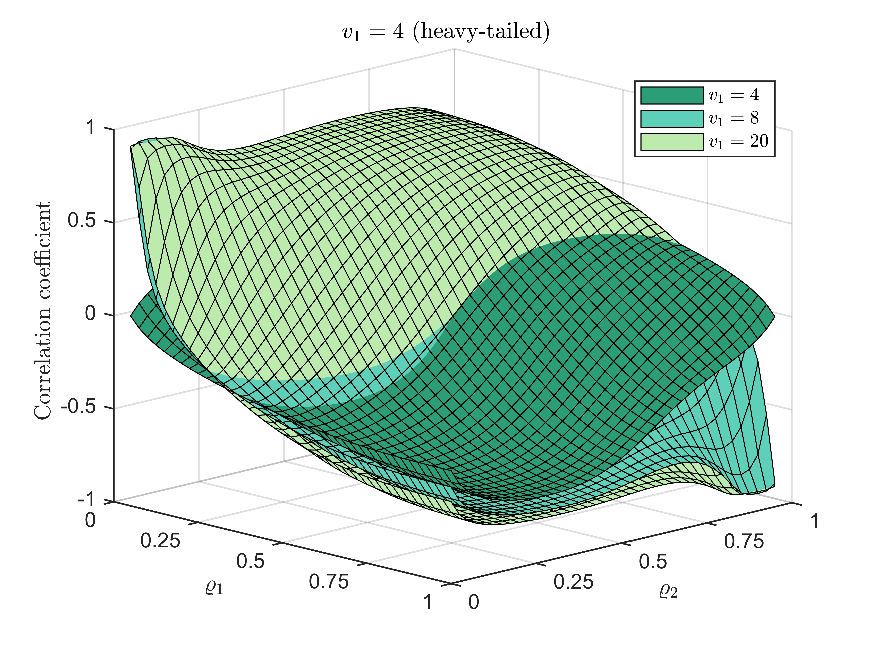}
        \label{qsdsta}
    }\hspace{10pt}
    \subfigure[Fixed $v_1=20$, $\varrho_1$ and $\varrho_2$ change in the case of $v_2=4$, $8$ and $20$.]{
        \centering
        \includegraphics[scale=0.38]{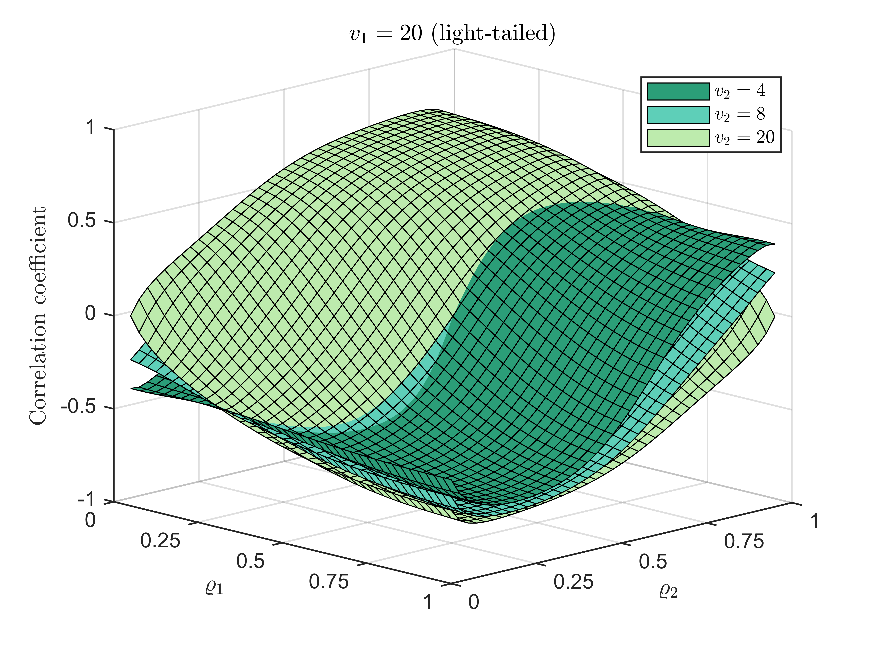}
        \label{qsdsta_light}
    }
    \caption{The correlation coefficient of QSD-ST models: fixed distribution of $\varepsilon_n$ is heavy-tailed or light-tailed. \label{qsdst_a}}
\end{figure}

\section{Approximating correlated volatility diffusions with
QSD models}\label{QAML}

According to the main results, QSD models based on two different asymmetric noises weakly converge to correlated volatility diffusions under a set of scaling rate conditions, as outlined in Assumption \ref{a1}, thus serving as an approximation. A natural application is that we can estimate and filter the diffusion processes by its counterpart QSD models. The statistical inference for the latter is standard, and they are easy to estimate using maximum likelihood estimation \citep{blasques2023quasi}. This forms the basis of the QAML method, which has been advocated, e.g., by \cite{barone2005option, fornari2006approximating, stentoft2011american, hafner2017weak} among others. 

In this section, a Monte Carlo experiment is designed to verify its ability to estimate and filter the correlated volatility diffusions. Specifically, we assume that the DGP is the following stochastic volatility process: 

\begin{equation}\label{DGP}
    \left\{\begin{aligned}
        &\dd X_t=\sigma_t\dd W_t^{(1)},\\
        &\dd \sigma_t^2=\left(\omega-\theta\sigma_t^2\right)\dd t+\kappa\sigma_t^2\dd W_t^{(2)},\\
        &\mbox{Cov}(\dd W_t^{(1)}, \dd W_t^{(2)})=\rho \dd t.
    \end{aligned}\right.
\end{equation}
When $\rho=0$, it can be inferenced based on SD models as in \cite{buccheri2021continuous}. However it fails when $\rho\neq 0$, because the two Brownian motion are independent in the continuous-time limit of SD model. 

We simulate $N=1000$ paths of \eqref{DGP} using the Euler–Maruyama discretization, with a time step of $\Delta t=1/19656$. Although this produces all the data, we can only sample them at certain frequencies, such as daily, weekly, or monthly, and at higher frequencies such as hourly or minutely. Therefore, we set the length between successive observations as $s\in\{6, 12, 78, 390, 1716, 4914\}$ and the time interval as $h=s\Delta t$. If the unit of time is years, this roughly corresponds to taking observations every thirty minutes, one hour, one day, one week, one month, and one quarter \footnote{For example, assuming there are 252 trading days in a year, with 6.5 hours of trading each day.} . 

In the experiment, the parameters of DGP are chosen as $\omega = 0.01$, $\theta = 0.2$, $\kappa = 2.5$, $\rho=-0.5$ and QSD-ST model \eqref{QSD-ST} are selected to approximate this diffusion. For fixed length $s$, we generate $s/\Delta t$ data points for one path to ensure that the same 19,656 observations are available for each frequency. We then estimate the parameters of the QSD-ST model based on these observed samples using the maximum likelihood estimation method, 
\[\Theta=\{\omega_h, \beta_h, \alpha_h, \varrho_1, v_1, \varrho_2, v_2\}.\]
Next, based on the convergence condition and continuous-time limit form \eqref{coQSD-T} and \eqref{formrho}, the diffusion parameters $(\omega, \theta, \kappa, \rho)$ are recovered from $\Theta$. Finally, the volatility is filtered and compared with the true values generated from \eqref{DGP} in the sense of root mean square error (RMSE). For comparison, we also performed estimation and filtering using GARCH, t-GARCH, Beta-t, Beta-st and QSD-T models. Table \ref{results} reports the results. For each parameter, the mean and standard deviation of the estimates over $N$ paths are shown. For models involving only one t-distribution, $v=v_1, \varrho=\varrho_1$ are shown in table. Each RMSE is normalized by the RMSE of the GARCH model.

\begin{table}[http]
    \centering
    \caption{Recovered parameters $(\omega, \theta, \kappa, \rho)$, model parameters $(\varrho_i,v_i)$ and filter RMSE}
    \label{results}
    \footnotesize
    \begin{tabular}{cccccccc}
        \toprule
        &  & $s=6$ & $s=12$ & $s=78$ & $s=390$ & $s=1716$ & $s=4914$ \\
        \midrule
        \multirow{2}{*}{$\omega$}  & \multirow{8}{*}{QSD-ST} & 0.0157 & 0.0146 & 0.0157 & 0.0172 & 0.0174 & 0.0166 \\
                & & (0.0077) & (0.0039) & (0.0020) & (0.0013) & (0.0009) & (0.0008) \\
        \multirow{2}{*}{$\theta$} & & 18.3323 & 13.5450 & 6.4293 & 3.6777 & 2.1090 & 1.3149 \\
               & & (1.7952) & (1.3330) & (0.5097) & (0.2322) & (0.0963) & (0.0533) \\
        \multirow{2}{*}{$\kappa$} &  & 2.5890 & 2.6275 & 2.7219 & 2.8291 & 2.6527 & 2.2198 \\
               & & (0.1767) & (0.1721) & (0.1433) & (0.1785) & (0.1549) & (0.1447) \\
        \multirow{2}{*}{$\rho$}  & & -0.5229 & -0.5240 & -0.5253 & -0.5162 & -0.5153 & -0.5113 \\
               & & (0.0126) & (0.0114) & (0.0100) & (0.0112) & (0.0105) & (0.0146) \\
        \hline
        \multirow{4}{*}{$\varrho_1$} & QSD-ST  & 0.5063 & 0.5094 & 0.5219 & 0.5405 & 0.5656 & 0.5861 \\
                 &  & (0.0053) & (0.0052) & (0.0049) & (0.0045) & (0.0040) & (0.0037) \\
                 & Beta-st  & 0.4994 & 0.5010 & 0.5078 & 0.5220 & 0.5461 & 0.5678 \\
                 &  & (0.0056) & (0.0055) & (0.0049) & (0.0044) & (0.0039) & (0.0036) \\
        \hline
        \multirow{10}{*}{$v_1$}   & QSD-ST  & 35.7464 & 25.5205 & 11.6330 & 6.7116 & 4.3703 & 3.3638 \\
                 &  & (9.4181) & (4.3414) & (0.8394) & (0.2699) & (0.1191) & (0.0734) \\
                 & t-GARCH  & 7.2876 & 6.6948 & 5.3601 & 6.5505 & 3.4941 & 3.0487 \\
                 &  & (2.6263) & (2.2321) & (0.9636) & (0.2580) & (0.1385) & (0.0888) \\
                 & Beta-t  & 158.8462 & 125.7032 & 18.4558 & 7.3869 & 4.2986 & 3.3416 \\
                 &  & (64.7877) & (52.5941) & (3.7351) & (0.4260) & (0.1531) & (0.1026) \\
                 & Beta-st  &160.9632  & 126.3390 & 18.7387  & 7.3456 & 4.2130 & 3.2232 \\
                 &  & (59.7821) & (56.9696) & (4.5420) & (0.4198) & (0.1454) & (0.0959) \\
                 & QSD-T  & 22.8392 & 19.5691 & 9.7865 & 6.5384 & 4.3801 & 3.4007 \\
                 &  & (6.9123) & (5.3643) & (0.4749) & (0.2188) & (0.1038) & (0.0703) \\
        \hline
        \multirow{2}{*}{$\varrho_2$} & QSD-ST & 0.1219 & 0.1263 & 0.1494 & 0.1986 & 0.2636 & 0.3289 \\
                 &  & (0.0098) & (0.0092) & (0.0096) & (0.0114) & (0.0112) & (0.0155) \\
        \hline
        \multirow{4}{*}{$v_2$}   & QSD-ST  & 73.3495 & 68.5080 & 57.2311 & 68.6815 & 64.6205 & 68.9667 \\
                 &  & (42.6400) & (35.4345) & (18.9173) & (24.8915) & (9.0309) & (8.1920) \\
                 & QSD-T  & 36.5237 & 42.5638 & 255.5882 & 231.2275 & 197.3231 & 170.6124 \\
                 &  & (13.9576) & (17.6468) & (18.1998) & (15.3321) & (13.8647) & (17.7997) \\
        \hline
        \multirow{6}{*}{RMSE}    
        & GARCH     & 1.0000   & 1.0000   & 1.0000   & 1.0000   & 1.0000   & 1.0000\\
        & QSD-ST   & 0.4835   & 0.5251   & 0.6003   & 0.6917   & 0.8679   & 0.8832   \\
        & t-GARCH   & 1.1524   & 1.1661   & 1.0748   & 1.1004   & 1.1065   & 0.9974   \\
        & Beta-t   & 0.9779   & 0.7817   & 0.7794   & 0.8396   & 0.9703   & 0.9517   \\
        & Beta-st   & 0.7734   & 0.8060   & 0.7589   & 0.8548   & 0.9953   & 0.9303   \\
        & QSD-T   & 0.5527   & 0.6637   & 0.7704   & 0.7835   & 0.8741   & 0.9050   \\
        \bottomrule
    \end{tabular}
    
\end{table}

It can be observed that the QAML estimates of $\theta$ are inconsistent, confirming the findings of \cite{wang2002asymptotic}. Similarly, the simulation results in \cite{hafner2017weak} and \cite{buccheri2021continuous} also demonstrate significant overestimation of $\theta$. In fact, we recover $\theta$ through the relation $\theta = (1 - \beta_h)/h$. However, we observe that as $h$ decreases, the estimates of $\beta_h$ do not sufficiently approach 1, resulting in an inflated ratio. But except for $\theta$, the other parameters are close to the true values under any fixed frequency, and the standard deviation is even smaller at low frequency (large $s$). We need to pay special attention to the correlation coefficient $\rho$, which is closely related to the parameters $(\varrho,v)$ of two skew t-distributions. It can be seen that the estimation of skewness parameter $\varrho_1>0.5$, and $\varrho_2<0.5$, which exactly describes a negative correlation, i.e., the leverage effect. Additionally, as the frequency decreases, $\varrho_1$ increases, while the degrees of freedom $v_1$ decreases. It implies that as the likelihood of discrete-time observations of the stochastic volatility \eqref{DGP} becomes more non-normal and asymmetric with time aggregation, the QSD-ST model captures the dynamics better through robust left-skewed and fat-tailed estimates. As the frequency increases, the estimated degrees of freedom gradually increase, because over short time periods, asset returns can be considered as a It\^o integral of constant volatility, following a normal distribution. It is important to note, however, that unlike the GARCH model, the score driving the volatility comes from an asymmetric distribution, even if its tail resembles that of a normal distribution. This is the key to generating negative correlation.

For the other models, their estimated values of $v$ (t-GARCH, Beta-t, Beta-st, QSD-T) and $\varrho$ (Beta-st) follow a consistent trend as $s$ increases, with $v$ decreasing and $\varrho$ increasing. This again verifies the conclusion that asset returns aggregate over time, showing fat-tails and left-skewness. Notably, the QSD-ST model's filter achieves the lowest RMSE across all frequencies, especially at higher frequencies. This is because as the frequency increases, the QSD-ST model gradually approximates a correlated volatility diffusion, while the other models do not. This further explains empirical findings of \cite{blasques2023quasi} from another perspective: why the QSD-ST model outperforms the SD model when estimating and filtering empirical data with leverage effects.

\section{Conculsion}\label{Conclusion}
The SD model closely links the shape of the conditional distribution of innovations to the design of the updating equation for time-varying parameters, whereas the QSD model breaks this connection. It directly leads to the emergence of correlation between the two Brownian motions in its continuous-time limit. 

Specifically, we examines the continuous-time limit when the QSD model is used to describe volatility, and the loss function is chosen as the log-likelihood of another scale family distribution, extending the continuous-time limit of \cite{buccheri2021continuous} on the SD model. We find that the limit is a stochastic volatility diffusion, where two Brownian motions are correlated. This correlation is closely tied to the two distributions that drive the innovations and compute the score. When these two distributions are the same (i.e., the SD model) or both are symmetric, the correlation vanishes. 

Through examples of the QSD-T and QSD-ST models, we specifically demonstrate how the choice of distribution affects key parameters in the diffusion limit, namely the diffusion coefficient and correlation coefficient. Finally, we employ the QSD model to approximate correlated volatility diffusion. Experimental results show that although its QAML estimates are not consistent, it can roughly recover diffusion parameters even with low-frequency data. The comparison with other models indicates that, with time aggregation, data generated by correlated volatility diffusion exhibit fat-tails and left-skewness, where the QSD-ST model provides the best filtering performance.

\newpage
\appendix
\section{Weak convergence of Markov processes to diffusion}\label{appendix}
\subsection{Set-up}\label{appendix1}
Let $\{X^{(h)}_{kh}\}_{k\in \mb{N}}$ be a $\mb{R}^d$-valued discrete-time Markov process, which has a timestamp of length $h\in \mb{R}^+$. $\mathcal{F}_{kh}=\sigma(X^{(h)}_{kh}, X^{(h)}_{(k-1)h},\ldots, X^{(h)}_0,kh)$ is its $\sigma$-algebra. Let $\{P_{kh}^{(h)}\}_{k\in \mb{N}}: (\mb{R}^d, \mathcal{B}(\mb{R}^d))\to [0,1]$ be the family of one step transition function of $\{X^{(h)}_{kh}\}_{k\in \mb{N}}$ and $\nu_h$ be a probability measure on $(\mb{R}^d,\mathcal{B}(\mb{R}^d))$, denoting the initial distribution of $X_0^{(h)}$. Let $D([0,\infty),\mb{R}^{d})$ be the space of càdlàg functions from $[0,\infty)$ to $\mb{R}^{d}$ equipped with the Skorokhod topology. Now we construct a continuous-time process $X^{(h)}=\{X_t^{(h)}\}_{t\geq 0}$ taking values in $D([0,\infty), \mb{R}^{d})$ based on $\{X^{(h)}_{kh}\}_{k\in \mb{N}}$, and let $\mb{P}^{(h)}$ be its probability measure, satisfying
\[\begin{split}
    &\mb{P}^{(h)}\left[X_0^{(h)}\in B\right]=\nu_h(B), \quad \forall B\in\mathcal{B}(\mb{R}^d), \\
    &\mb{P}^{(h)}\left[\left.X_{(k+1)h}^{(h)}\in B \right|\mathcal{F}_{kh}\right]=P_{kh}^{(h)}(X_{kh}^{(h)}, B), \\
    &\mb{P}^{(h)}\left[X_t^{(h)}=X_{kh}^{(h)}, kh\leq t<(k+1)h\right]=1.
\end{split}\]
Intuitively, $X_t^{(h)}$ can be seen as a continuous-time process obtained by extending the values of $X_{kh}^{(h)}$ to the interval $[kh,(k+1)h)$. 

\begin{remark}
    Note that this extension is not linear but rather stepwise, thus $X^{(h)}$ taking values in $D([0,\infty),\mb{R}^{d})$ rather than a smaller space $C([0,\infty),\mb{R}^{d})$, the continuous function space. It resulting in weaker conditions regarding initial values when $X^{(h)}$ weakly converge in $D$. Specifically, while linear interpolation necessitates $X_0^{(h)}=X_0$ for all $h$, the latter only requires $X_0^{(h)}\tod X_0$ as $h\to 0$, where ``$\tod$'' denotes converges in distribution. For further details, refer to \cite{ethier1986markov}. 
\end{remark}

Here we summarize and distinguish the four processes that have mentioned above, which simultaneously clarify how we transition from discrete-time models to continuous-time models: 
    \begin{enumerate}[label=(\roman*)]
        \item The discrete-time process (model) $X_n$ we are interested;
        \item The discrete-time process $X_{kh}^{(h)}$ associated with time interval $h$, and taking values only at time $0,h,2h\ldots$;
        \item The continuous-time process $X_{t}^{(h)}$ constructed based on $X_{kh}^{(h)}$, whose paths are step functions taking jumps at time $0,h,2h\ldots$;
        \item The continuous-time process $X_t$, which is the weak convergence limit of $X_t^{(h)}$ as $h\to0$.
    \end{enumerate} 
We begin with the process $X_n$, and as the observation interval shrinks, we obtain $X_{kh}^{(h)}$. Then, we extend it to the continuous-time version $X_{t}^{(h)}$ by left endpoint extension. Finally, we investigate the weak convergence of $X_{t}^{(h)}$ as random elements in the space $D([0,\infty),\mb{R}^{d})$. 

\subsection{Weak convergence theorem}\label{appendix2}
We present a set of sufficient conditions for the weak convergence of a sequence of Markov processes $X^{(h)}$, indexed by the time interval $h$, to a diffusion process $X=\{X_t\}_{t\geq 0}$ as $h\to0$. This result is mainly derived from \cite{stroock1997multidimensional} and \cite{nelson1990arch}. Before proceeding, we define the following quantity:
\[\begin{split}
    &b_h(x,t):= h^{-1}\mb{E}\left[\left.X_{t+h}^{(h)}-X_{t}^{(h)}\right|X_{t}^{(h)}=x\right],\\
    &a_h(x,t):= h^{-1} \mb{E}\left[\left.\left(X_{t+h}^{(h)}-X_{t}^{(h)}\right)\left(X_{t+h}^{(h)}-X_{t}^{(h)}\right)^\TT\right|X_{t}^{(h)}=x\right], \\
    &c_{h,i,\delta}(x,t):= h^{-1}\mb{E}\left[\left.\left|\left \langle X_{t+h}^{(h)}-X_{t}^{(h)}, e_i\right \rangle\right|^{2+\delta}\right|X_{t}^{(h)}=x\right]. 
\end{split}\]
Where, $\mb{E}[\cdot]$ denotes the expectation corresponding to the probability measure $\mathbb{P}^{(h)}$, the superscript $(h)$ is omitted without causing ambiguity. $\langle\cdot,\cdot\rangle$ represents the natural inner product in $\mathbb{R}^d$, and $\{e_1, e_2, \ldots, e_d\}$ are its unit vectors.

There are four conditions that guarantee the convergence:

\begin{condition}\label{con1}
    There exists a continuous function $a(x,t):\mb{R}^d\times[0,\infty)\to M_{d\times d}$, the space of $d\times d$ nonnegative definite symmetric matrices, and a continuous function $b(x,t):\mb{R}^d\times[0,\infty)\to \mb{R}^{d}$, such that for all $ R>0, T>0$, 
    \begin{equation}\label{con11}
        \lim_{h\to0}\sup_{\|x\|\leq R, t\leq T}\|b_h(x,t)-b(x,t)\|=0,
    \end{equation}
    \begin{equation}\label{con12}
        \lim_{h\to0}\sup_{\|x\|\leq R, t\leq T}\|a_h(x,t)-a(x,t)\|=0.
    \end{equation}
    Furthermore, there exists a $\delta>0$, for all $R>0, T>0, i=1,2,\ldots,d$ such that
    \begin{equation}\label{con13}
        \lim_{h\to0}\sup_{\|x\|\leq R, t\leq T}c_{h,i,\delta}(x,t)=0.
    \end{equation}
\end{condition}
\begin{condition}\label{con2}
    There exists a continuous function $\sigma(x,t):\mb{R}^d\times[0,\infty)\to V_{d\times d}$, the space of $d\times d$ matrices, such that $a(x,t)=\sigma(x,t)\sigma(x,t)^\TT$.
\end{condition}
\begin{condition}\label{con3}
    As $h\to0$, $X_0^{(h)} \tod X_0$, which has the probability measure $\nu_0$ on $(\mb{R}^d,\mathcal{B}(\mb{R}^d))$.
\end{condition}
\begin{condition}\label{con4}
    There exists a unique weak solution to the following SDE defined by $b(x,t)$ in condition \ref{con1}, $\sigma(x,t)$ in condition \ref{con2} and initial probability measure $\nu_0$ in condition \ref{con3}:
    \begin{equation}\label{sde}
        \left\{\begin{aligned}
            &\dd X_t=b(X_t,t)\dd t+\sigma(X_t,t)\dd W_t,\\
            &\mb{P}(X_0\in B)=\nu_0(B),  \quad \mbox{for any}\ B\in\mathcal{B}(\mb{R}^d), 
        \end{aligned}\right.
    \end{equation}
    where $W_t$ is an $d$-dimensional standard Brownian motion.
\end{condition}
\begin{theorem}[Nelson, 1990]\label{corethm}
    Under conditions \ref{con1}--\ref{con4}, $X_t^{(h)}$ weakly converges to $X_t$, which is the unique weak solution of SDE \eqref{sde}. 
\end{theorem}

\section{Proof of Theorem \ref{main thm}}\label{appendixB}

The main thing we need do is to verify the four conditions in Theorem \ref{corethm}. 

    {\it Condition} \ref{con1}.\eqref{con11}:
    From \eqref{discrete-QSD}, the increment of $Z_t^{(h)}=(X_t^{(h)},\lambda_{t+h}^{(h)})$ in unit of time interval $h$ is given by 
    \[\begin{split}
        &X_{t+h}^{(h)}-X_{t}^{(h)}=c_h\lambda_{t+h}^{(h)}+y_{t+h}^{(h)}=c_h\lambda_{t+h}^{(h)}+\sqrt{\varphi(\lambda_{t+h}^{(h)})}\varepsilon_{t+h}^{(h)},\\
        &\lambda_{t+2h}^{(h)}-\lambda_{t+h}^{(h)}=\omega_h+(\beta_h-1)\lambda_{t+h}^{(h)}+\alpha_{h}S(\lambda_{t+h}^{(h)})\nablaq_{t+h}^{(h)}.
    \end{split}
    \]
    Note that $\lambda_{t+h}^{(h)}$ is $\mathcal{F}_{t}^{(h)}$-measurable, so we can take it out of the conditional expectation. Since the fact that innovation $y$ and quasi score $\nablaq$ have zero mean in our setting, we have  
    \[\begin{split}
        &h^{-1}\mb{E}\left[\left.X_{t+h}^{(h)}-X_{t}^{(h)}\right|\mathcal{F}_{t}^{(h)}\right]=h^{-1}c_h\lambda_{t+h}^{(h)},\\
        &h^{-1}\mb{E}\left[\left.\lambda_{t+2h}^{(h)}-\lambda_{t+h}^{(h)}\right|\mathcal{F}_{t}^{(h)}\right]=h^{-1}\omega_h-h^{-1}(1-\beta_h)\lambda_{t+h}^{(h)}.
    \end{split}\]
    Therefore, for any $z = (x, \lambda)\in \mb{R}\times \Gamma, t\geq 0$, by assumption \ref{a1}, 
    \[b_h(z,t)=\frac{1}{h}\begin{bmatrix}c_h\lambda\\ \omega_h-(1-\beta_h)\lambda
    \end{bmatrix}\stackrel{h\to0}\longrightarrow\begin{bmatrix}c\lambda\\\omega-\theta\lambda
    \end{bmatrix}=b(z,t).\]
    Clearly, $b(z,t)$ is continuous and the convergence is uniform.

    {\it Condition} \ref{con1}.\eqref{con12}: In fact, we don't need to expand the squared increment completely, because many terms involving $\varepsilon_{t+h}^{(h)}$ or $\nablaq_{t+h}^{(h)}$ can be dropped when taking the expectation. Recall that $h^{-1/2}\varepsilon_{t+h}^{(h)}\big|\mathcal{F}_{t}^{(h)}\stackrel{d}{=}U\stackrel{d}{\sim}f(\cdot)$, we have
    \[\begin{split}
        &h^{-1}\mb{E}\left[\left.\left(X_{t+h}^{(h)}-X_{t}^{(h)}\right)^2\right|\mathcal{F}_{t}^{(h)}\right]=h^{-1}\mb{E}[c_h^2(\lambda_{t+h}^{(h)})^2+(y_{t+h}^{(h)})^2|\mathcal{F}_{t}^{(h)}]\\
        &=h^{-1}c_h^2(\lambda_{t+h}^{(h)})^2+\mb{E}\left[\left(\sqrt{\varphi(\lambda_{t+h}^{(h)})}h^{-1/2}\varepsilon_{t+h}^{(h)}\right)^2\bigg|\mathcal{F}_{t}^{(h)}\right]=h^{-1}c_h^2(\lambda_{t+h}^{(h)})^2+\varphi(\lambda_{t+h}^{(h)})\eta,
    \end{split}\]
    \[\begin{split}
        &h^{-1}\mb{E}\left[\left.\left(\lambda_{t+2h}^{(h)}-\lambda_{t+h}^{(h)}\right)^2\right|\mathcal{F}_{t}^{(h)}\right]\\
        &=h^{-1}\mb{E}\left[\omega_h^2+(1-\beta_h)^2(\lambda_{t+h}^{(h)})^2+\alpha_{h}^2 S^2(\lambda_{t+h}^{(h)})(\nablaq_{t+h}^{(h)})^2-2\omega_h(1-\beta_h)\lambda_{t+h}^{(h)}\big|\mathcal{F}_{t}^{(h)}\right]\\
        &=h^{-1}\omega_h^2+h^{-1}(1-\beta_h)^2(\lambda_{t+h}^{(h)})^2+h^{-1}\alpha_{h}^2 S^2(\lambda_{t+h}^{(h)})\gamma(\lambda^{(h)}_{t+h})-2h^{-1}\omega_h(1-\beta_h)\lambda_{t+h}^{(h)},
    \end{split}\]

    \[\begin{split}
    &h^{-1}\mb{E}\left[\left.\left(X_{t+h}^{(h)}-X_{t}^{(h)}\right)\left(\lambda_{t+2h}^{(h)}-\lambda_{t+h}^{(h)}\right)\right|\mathcal{F}_{t}^{(h)}\right]\\
    &=h^{-1}\mb{E}\left[\left(c_h\lambda_{t+h}^{(h)}+\sqrt{\varphi(\lambda_{t+h}^{(h)})}\varepsilon_{t+h}^{(h)}\right)\left(\omega_h-(1-\beta_h)\lambda_{t+h}^{(h)}+\alpha_{h}S(\lambda_{t+h}^{(h)})\nablaq_{t+h}^{(h)}\right)\big|\mathcal{F}_{t}^{(h)}\right]\\
    &=h^{-1}\left[c_h\omega_h\lambda_{t+h}^{(h)}-c_h(1-\beta_h)(\lambda_{t+h}^{(h)})^2\right]+h^{-1}\sqrt{\varphi(\lambda_{t+h}^{(h)})}\alpha_{h}S(\lambda_{t+h}^{(h)})\mb{E}\left[\varepsilon_{t+h}^{(h)}\nablaq_{t+h}^{(h)}\big|\mathcal{F}_{t}^{(h)}\right].
    \end{split}
    \]
    Focusing on the last term in the previous equation, and substituting the result of Lemma \ref{nablalem}, we have
    \[\begin{split}
        &h^{-1}\sqrt{\varphi(\lambda_{t+h}^{(h)})}\alpha_{h}S(\lambda_{t+h}^{(h)})\mb{E}\left[\varepsilon_{t+h}^{(h)}\nablaq_{t+h}^{(h)}\big|\mathcal{F}_{t}^{(h)}\right]\\
        &=h^{-1}\sqrt{\varphi(\lambda_{t+h}^{(h)})}\alpha_{h}S(\lambda_{t+h}^{(h)})\frac{-\varphi'(\lambda_{t+h}^{(h)})}{2\varphi(\lambda_{t+h}^{(h)})}\sqrt{h}\int_{-\infty}^{\infty}\left[1+\frac{g'(u)}{g(u)}u\right]u f(u)\dd u\\
        &=h^{-1/2}\alpha_h S(\lambda_{t+h}^{(h)})\frac{-\varphi'(\lambda_{t+h}^{(h)})}{2\sqrt{\varphi(\lambda_{t+h}^{(h)})}}\left[\int_{-\infty}^{\infty}u f(u)\dd u+\int_{-\infty}^{\infty}\frac{u^2g'(u)}{g(u)}f(u)\dd u\right]\\
        &=h^{-1/2}\alpha_h S(\lambda_{t+h}^{(h)})\frac{-\varphi'(\lambda_{t+h}^{(h)})}{2\sqrt{\varphi(\lambda_{t+h}^{(h)})}}\rho.
    \end{split}
    \]
    According to assumption \ref{a1}, $c_h, \omega_h, 1-\beta_h$ all converge to zero at the order of $h$. Terms such as $h^{-1}c_h^2(\lambda_{t+h}^{(h)})^2$, $h^{-1}\omega_h^2$, $h^{-1}(1-\beta_h)^2(\lambda_{t+h}^{(h)})^2$ vanish uniformly on $\Gamma$ as $h\to0$. Therefore, for any $z = (x, \lambda)\in \mb{R}\times \Gamma, t\geq 0$, 
    \[a_h(z,t)\stackrel{h\to0}\longrightarrow\begin{bmatrix}
        \varphi(\lambda)\eta & \alpha S(\lambda)\frac{-\varphi'(\lambda)}{2\sqrt{\varphi(\lambda)}}\rho\\
         \alpha S(\lambda)\frac{-\varphi'(\lambda)}{2\sqrt{\varphi(\lambda)}}\rho & \alpha^2 S^2(\lambda)\gamma(\lambda)
    \end{bmatrix}=a(z,t).\]
    The differentiability of $\varphi(\lambda)$ implies its continuity, and combining assumption \ref{a3}, it can be obtained $S(\lambda)\varphi'(\lambda)/\sqrt{\varphi(\lambda)}$ and $S^2(\lambda)\gamma(\lambda)=m_2\left[S(\lambda)\varphi'(\lambda)/\varphi(\lambda)\right]^2$ are continuous on $\Gamma$. Therefore, $a(z,t)$ is continuous, and clearly $\{a_h(z,t)\}_h$ are uniformly equicontinuous, which guarantee the convergence is uniform. 

    {\it Condition} \ref{con1}.\eqref{con13}: We will utilize the inequality
    \[\bigg|\sum_{i=1}^{n}a_i\bigg|^p\leq n^{p-1}\sum_{i=1}^{n}|a_i|^p, \quad \mbox{for all} \ a_i\in\mb{R}, \ p\geq 1, \]
    which can be derived by Jensen inequality with respect to $|\cdot|^p$. By assumption \ref{a2}, there exists a $\delta>0$ such that $\mb{E}|U|^{2+\delta}<\infty$, $\mb{E}[|\nablaq|^{2+\delta}|\mathcal{F}]<\infty$. Then, for any $\lambda_{t+h}^{(h)}\in D$, 
    \[
    \begin{split}
        0\leq&h^{-1}\mb{E}\left[\left.\left|X_{t+h}^{(h)}-X_{t}^{(h)}\right|^{2+\delta}\right|\mathcal{F}_{t}^{(h)}\right]=h^{-1}\mb{E}[|c_h\lambda_{t+h}^{(h)}+y_{t+h}^{(h)}|^{2+\delta}|\mathcal{F}_{t}^{(h)}]\\
        &\leq \frac{2^{1+\delta}|c_h\lambda_{t+h}^{(h)}|^{2+\delta}}{h}+\frac{2^{1+\delta}}{h}\mb{E}\left[\left|h^{1/2}\sqrt{\varphi(\lambda_{t+h}^{(h)})}h^{-1/2}\varepsilon_{t+h}^{(h)}\right|^{2+\delta}\big|\mathcal{F}_{t}^{(h)}\right]\\
        &=\frac{2^{1+\delta}|c_h\lambda_{t+h}^{(h)}|^{2+\delta}}{h}+h^{\delta/2}|\varphi(\lambda_{t+h}^{(h)})|^{1+\frac{\delta}{2}}\mb{E}|U|^{2+\delta}\longrightarrow 0,\ \mbox{as}\ h\to0.
    \end{split}
    \]
    \[
    \begin{split}
    0\leq&h^{-1}\mb{E}\left[\left.\left|\lambda_{t+2h}^{(h)}-\lambda_{t+h}^{(h)}\right|^{2+\delta}\right|\mathcal{F}_{t}^{(h)}\right]\\
    &=h^{-1}\mb{E}\left[\left|\omega_h-(1-\beta_h)\lambda_{t+h}^{(h)}+\alpha_{h}S(\lambda_{t+h}^{(h)})\nablaq_{t+h}^{(h)}\right|^{2+\delta}\big|\mathcal{F}_{t}^{(h)}\right]\\
    &\leq\frac{3^{1+\delta}}{h}\left[|\omega_h|^{2+\delta}+\left|(1-\beta_h)\lambda_{t+h}^{(h)}\right|^{2+\delta}+\left|\alpha_{h}S(\lambda_{t+h}^{(h)})\right|^{2+\delta}\mb{E}\left(\left|\nablaq_{t+h}^{(h)}\right|^{2+\delta}|\mathcal{F}_t^{(h)}\right)\right]\\
    &\longrightarrow 0,\ \mbox{as}\ h\to0.
    \end{split}
    \]
    Clearly, \eqref{con13} holds since $\varphi(\lambda)$ and $S(\lambda)$ is continuous.

    {\it Condition} \ref{con2}: We assert that such a $\sigma(z,t)$ exists because $a(z,t)$ is a symmetric positive definite matrix which admits a Cholesky decomposition. In fact, we observe that $\varphi(\lambda)\eta>0$ and the second-order sequential principal minor
    \[\begin{split}
        &\alpha^2 S^2(\lambda)\gamma(\lambda)\varphi(\lambda)\eta-\alpha^2 S^2(\lambda)\frac{[\varphi'(\lambda)]^2}{4\varphi(\lambda)}\rho^2\\
        =&\alpha^2 S^2(\lambda)\left[m_2\left(\frac{\varphi'(\lambda)}{\varphi(\lambda)}\right)^2\varphi(\lambda)\eta-\frac{[\varphi'(\lambda)]^2}{4\varphi(\lambda)}\rho^2\right]\\
        =&\alpha^2 S^2(\lambda)\frac{[\varphi'(\lambda)]^2}{\varphi(\lambda)}(m_2\eta-\frac{\rho^2}{4})>0,
    \end{split}\]
    Furthermore, we can explicitly calculate a $\sigma(z,t)$ satisfies the condition. Suppose lower triangualr matrix $\sigma(z,t)=\begin{bmatrix}
        \sigma_1 & 0\\
         \sigma_3 & \sigma_4
       \end{bmatrix}$, satisfying $\sigma(z,t)\sigma(z,t)^\TT=a(z,t)$, i.e.
       \begin{equation}\label{sigma_i}
        \begin{bmatrix}
            \sigma_1^2 & \sigma_1\sigma_3\\
            \sigma_1\sigma_3 & \sigma_3^2+\sigma_4^2
        \end{bmatrix}=\begin{bmatrix}
            \varphi(\lambda)\eta & \alpha S(\lambda)\frac{-\varphi'(\lambda)}{2\sqrt{\varphi(\lambda)}}\rho\\
            \alpha S(\lambda)\frac{-\varphi'(\lambda)}{2\sqrt{\varphi(\lambda)}}\rho & \alpha^2 S^2(\lambda)\gamma(\lambda)
        \end{bmatrix}.
    \end{equation}
    Then we have 
    \begin{equation}\label{sigma}
        \sigma_1=\sqrt{\varphi(\lambda)\eta}, \quad \ \sigma_3=\alpha S(\lambda)\frac{-\varphi'(\lambda)}{2\varphi(\lambda)\sqrt{\eta}}\rho, \quad \ \sigma_4=\alpha S(\lambda)\frac{\varphi'(\lambda)}{\varphi(\lambda)}\sqrt{m_2-\frac{\rho^2}{4\eta}}.
    \end{equation}
    Clearly, such $\sigma(z,t)$ is a continuous function satisfies condition \ref{con2}. 

    {\it Condition} \ref{con3} is satisfied due to $Z_0^{(h)}\tod Z_0$, and we now turn to {\it Condition} \ref{con4}. In fact, $a(z,t)=a(\lambda)$, $b(z,t)=b(\lambda)$ are independent of $x,t$, thus SDEs \eqref{sde} is an It\^o type equation. As demonstrated in \cite{stroock1997multidimensional}, an unique weak solution exists for \eqref{sde} if $b(\lambda)$, $\sigma(\lambda)$, $a(\lambda)$ are continuous and bounded, and $a(\lambda)$ is uniformly positive definite. We are already show that $b(\lambda)$, $\sigma(\lambda)$, $a(\lambda)$ are continuous on compact set $\Gamma$, thus bounded. We need to prove there exists a $\zeta>0$ such that for all $\lambda\in \Gamma$, $a(\lambda)-\zeta I$ is positive definite, where  $I$ unit matrix. 

    Because $\varphi(\lambda)$, $S^2(\lambda)\gamma(\lambda)$, $[S(\lambda)\varphi'(\lambda)]^2/\varphi(\lambda)$ are all positive continuous functions on $\Gamma$, and $\eta m_2-\rho^2/4>0$, there exist $c_1, c_2, M>0$ such that for all $\lambda\in\Gamma$, 
    \begin{equation}\label{first-order}
        \varphi(\lambda)\eta>c_1, \quad \frac{[\varphi'(\lambda)]^2}{\varphi(\lambda)}\alpha^2S^2(\lambda)(\eta m_2-\frac{\rho^2}{4})>c_2, \quad \varphi(\lambda)\eta+\alpha^2S^2(\lambda)\gamma(\lambda)<M.
    \end{equation}
    Then we have  
    \begin{equation}\label{second-order}
        \begin{split}
            \frac{[\varphi'(\lambda)]^2}{\varphi(\lambda)}\alpha^2S^2(\lambda)(\eta m_2-\frac{\rho^2}{4})&>\frac{c_2}{M}[\varphi(\lambda)\eta+\alpha^2S^2(\lambda)\gamma(\lambda)]\\
            &>\frac{c_2}{M}[\varphi(\lambda)\eta+\alpha^2S^2(\lambda)\gamma(\lambda)]-\left(\frac{c_2}{M}\right)^2. 
        \end{split}
    \end{equation}
    Taking $\zeta=\min\left\{c_1, \frac{c_2}{M}\right\}$ guarantees $a(\lambda)-\zeta I$ is definite for all $\lambda\in \Gamma$. 
    
    According to Theorem \ref{corethm}, the following diffusion process is the continuous-time limit of QSD volatility models: 
    \begin{equation}\label{QSD-sde_1}
        \left\{\begin{aligned}
            &\dd X_t=c\lambda_t\dd t+\sigma_1\dd B_t^{(1)},\\
            &\dd \lambda_t=(\omega-\theta\lambda_t)\dd t+\sigma_3\dd B_t^{(1)}+\sigma_4\dd B_t^{(2)},\\
            &\mb{P}(Z_0\in B)=\nu_0(B),  \quad \mbox{for any} \ B\in\mathcal{B}(\mb{R}\times \Gamma),
        \end{aligned}\right.
    \end{equation}
    where $B_t^{(1)}, B_t^{(2)}$ are two independent Brownian motions, and $\sigma_1,\sigma_3,\sigma_4$ are given by \eqref{sigma}. Actually, the noise driving $\lambda_t$ is correlated with that driving $X_t$, so we transform \eqref{QSD-sde_1} into an equivalent process with a more intuitive form. Since $\sigma_1^2=\varphi(\lambda_t)\eta$, $\sigma_3^2+\sigma_4^2=\alpha^2 S^2(\lambda_t)\gamma(\lambda_t)$, we have
    \[\sigma_1 B_t^{(1)}\sim \sqrt{\varphi(\lambda_t)\eta}W_t^{(1)}, \quad \sigma_3 B_t^{(1)}+\sigma_4 B_t^{(2)}\sim \alpha S(\lambda_t)\sqrt{\gamma(\lambda_t)}W_t^{(2)}, \]
    where $W_t^{(1)}, W_t^{(2)}$ are two standard Brownian motions. Their covariance is given by  
    \[\begin{split}
        &\mbox{Cov}\left(W_t^{(1)},W_t^{(2)}\right)=\frac{\mbox{Cov}\left(\sigma_1 B_t^{(1)},\sigma_3 B_t^{(1)}+\sigma_4 B_t^{(2)}\right)}{\alpha S(\lambda_t)\sqrt{\gamma(\lambda_t)\varphi(\lambda_t)\eta}}\\
        &=\frac{\sigma_1\sigma_3\mbox{Var}(B_t^{(1)})}{\alpha S(\lambda_t)\sqrt{\gamma(\lambda_t)\varphi(\lambda_t)\eta}}=\frac{\alpha S(\lambda_t)\frac{-\varphi'(\lambda_t)}{2\sqrt{\varphi(\lambda_t)}}\rho t}{\alpha S(\lambda_t)\sqrt{\gamma(\lambda_t)\varphi(\lambda_t)\eta}}\\
        &=-\frac{\varphi'(\lambda_t)\rho t}{2\varphi(\lambda_t)\sqrt{\eta\gamma(\lambda_t)}}=-\frac{\rho t}{2\sqrt{\eta m_2}}.
    \end{split}\]
    In summary, as $h\to0$, $Z_t^{(h)}=\left(X_t^{(h)}, \lambda_{t+h}^{(h)}\right)$ weakly converges towards following SDE's unique weak solution: 
    \[\left\{\begin{aligned}
        &\dd X_t=c\lambda_t\dd t+\sqrt{\varphi(\lambda_t)\eta}\dd W_t^{(1)},\\
        &\dd \lambda_t=(\omega-\theta\lambda_t)\dd t+\alpha S(\lambda_t)\sqrt{\gamma(\lambda_t)}\dd W_t^{(2)},\\
        & \mbox{\rm Cov}\left[\dd W_t^{(1)},\dd W_t^{(2)}\right]=-\frac{\rho}{2\sqrt{\eta m_2}}\dd t, \\
        & \mb{P}(Z_0\in B)=\nu_0(B),  \quad \mbox{for any} \ B\in\mathcal{B}(\mb{R}\times \Gamma).
    \end{aligned}\right.\]

\section{Proof of Theorem \ref{main thm2}}\label{appendixC}
It is similar to the proof in the previous part, here we prove briefly. First, it can be derived from Lemma \ref{nablalem} that
    \[\begin{array}{l}
        \mb{E}(\widetilde{\nablaq}|\mathcal{F})=0,\\
        \mb{E}(\widetilde{\nablaq}^2|\mathcal{F})=(m_2-\mu^2)\left[\frac{\varphi'(\lambda)}{\varphi(\lambda)}\right]^2,\\
        \mb{E}(\widetilde{\nablaq}\varepsilon|\mathcal{F})=\mb{E}(\nablaq\varepsilon|\mathcal{F}).
    \end{array}
    \]
    It is clear that the calculations for $X_t$ are identical to those in the previous case. We proceed directly to consider the drift and second moment per unit time of $\lambda_t$.
    \[\begin{split}
        &h^{-1}\mb{E}\left[\left.\lambda_{t+2h}^{(h)}-\lambda_{t+h}^{(h)}\right|\mathcal{F}_{t}^{(h)}\right]\\
        &=h^{-1}\omega_h-h^{-1}(1-\beta_h)\lambda_{t+h}^{(h)}+h^{-1}\frac{\mu_h\alpha_h S(\lambda_{kh}^{(h)})\varphi'(\lambda_{kh}^{(h)})}{\varphi(\lambda_{kh}^{(h)})}\\
        &\longrightarrow \omega-\theta\lambda_t+\frac{\alpha\mu S(\lambda_t)\varphi'(\lambda_t)}{\varphi(\lambda_t)}, \quad \mbox{as} \ h\to0.
    \end{split}\]
    In fact, as $h\to0$, all terms in the second moment per unit of time vanish except for $[\alpha_{h}S(\lambda_{t+h}^{(h)})\widetilde{\nablaq}_{t+h}^{(h)}]^2$, so that 
    \[\begin{split}
        &h^{-1}\mb{E}\left[\left.\left(\lambda_{t+2h}^{(h)}-\lambda_{t+h}^{(h)}\right)^2\right|\mathcal{F}_{t}^{(h)}\right]\longrightarrow \alpha^2 S^2(\lambda_t)(m_2-\mu^2)\left[\frac{\varphi'(\lambda)}{\varphi(\lambda)}\right]^2, \quad \mbox{as} \ h\to0.
    \end{split}\]
    Therefore, we have
    \[\begin{split}
        &b(z,t)=\begin{bmatrix}c\lambda\\ \omega-\theta\lambda_t+\frac{\alpha\mu S(\lambda_t)\varphi'(\lambda_t)}{\varphi(\lambda_t)}
        \end{bmatrix},\\
        &a(z,t)=\begin{bmatrix}
            \varphi(\lambda)\eta & \alpha S(\lambda)\frac{-\varphi'(\lambda)}{2\sqrt{\varphi(\lambda)}}\rho\\
             \alpha S(\lambda)\frac{-\varphi'(\lambda)}{2\sqrt{\varphi(\lambda)}}\rho & \alpha^2 S^2(\lambda_t)(m_2-\mu^2)\left[\frac{\varphi'(\lambda)}{\varphi(\lambda)}\right]^2
        \end{bmatrix}.
    \end{split}\]
    Since there exists a $\delta>0$ such that 
    \[\mb{E}(|\widetilde{\nablaq}|^{2+\delta}|\mathcal{F})\leq 2^{1+\delta}\mb{E}(|\nablaq|^{2+\delta}|\mathcal{F})+2^{1+\delta}\left|\frac{\mu\varphi'(\lambda_{n})}{\varphi(\lambda_{n})}\right|^{2+\delta}<\infty, \]
    meanwhile, the second-order sequential principal minor 
    \[\alpha^2 S^2(\lambda)\frac{[\varphi'(\lambda)]^2}{\varphi(\lambda)}[(m_2-\mu^2)\eta-\frac{\rho^2}{4}]>0. \]
    By the same procedure as in the previous part, it can be shown that remaining conditions are satisfied. 
    
    Finally, we obtain a bivariate diffusion as the weak convergence limit and transform it into the form driven by two correlated Brownian motions. According to $a(z,t)$, the covariance between the two new Brownian morions is given by
    \[\mbox{Cov}\left(W_t^{(1)},W_t^{(2)}\right)=\frac{\alpha S(\lambda)\frac{-\varphi'(\lambda)}{2\sqrt{\varphi(\lambda)}}\rho t}{\alpha S(\lambda)\frac{\varphi'(\lambda)}{\varphi(\lambda)}\sqrt{(m_2-\mu^2)\varphi(\lambda_t)\eta}}=-\frac{\rho t}{2\sqrt{\eta (m_2-\mu^2)}}.\]
    Thus, the limit \eqref{new-QSD-sde} is obtained.

\newpage
\bibliographystyle{apalike}
\bibliography{ref}
\end{document}